\documentclass[reqno,11pt]{amsart}
\usepackage{amsfonts}
\usepackage{cite}
\usepackage{times}
\usepackage{amssymb, amsmath,latexsym}
\usepackage{color}
\usepackage{verbatim}
\allowdisplaybreaks

\theoremstyle{plain}
\newtheorem{theorem}{Theorem}[section]

\newtheorem{lemma}[theorem]{Lemma}
\newtheorem{proposition}[theorem]{Proposition}

\theoremstyle{definition}
\newtheorem{defn}{Definition}[section]
\theoremstyle{remark}
\newtheorem{rem}{Remark}[section]
\numberwithin{equation}{section}

\newcommand{\mbfu}{\mathbf{u}}
\newcommand{\mw}{\mathbf{w}}
\newcommand{\real}{\mathbb{R}}
\newcommand{\ds}{\displaystyle}

\definecolor{Green}{rgb}{0.010,0.7,0.02}

\newcommand{\NN}{\mathbb{N}}
\newcommand{\RR}{\mathbb{R}}
\newcommand{\pa}{\partial}

\newcommand{\loc}{\text{loc}}
\newcommand{\dive}{\mbox{div}}
\newcommand{\mbfv}{\mathbf{v}}
\newcommand{\curl}{\text{curl}}
\newcommand{\mPhi}{\boldsymbol{\Phi}}

\newcommand{\mxi}{\boldsymbol{\xi}}
\newcommand{\bomega}{\boldsymbol{\omega}}
\newcommand{\cL}{\mathcal{L}}

\begin{document}

\title[Planar limits of helically symmetric flows]
{Planar limits of three-dimensional incompressible  flows with helical symmetry}

\author[M.C. Lopes Filho {\em et al.}]{Milton C. Lopes Filho}
\address[M.C. Lopes Filho]
{Mathematics Institute, \\
Federal University of Rio de Janeiro,\\
P.O. Box 68530, 21941-909 
Rio de Janeiro, RJ,  Brazil}
\email{mlopes@im.ufrj.br}

\author[]{Anna L. Mazzucato}
\address[A.L. Mazzucato]{Department of Mathematics\\
Penn State University \\
University Park, PA, 16801, U.S.A.}
\email{alm24@psu.edu}

\author[]{Dongjuan Niu}
\address[D. Niu]
{School of Mathematical Sciences  \\
Capital Normal University \\
Beijing 100048, P. R. China} \email{niuniudj@gmail.com}

\author[]{Helena J. Nussenzveig Lopes}
\address[H.J. Nussenzveig Lopes]
{Mathematics Institute, \\
Federal University of Rio de Janeiro,\\
P.O. Box 68530, 21941-909 
Rio de Janeiro, RJ,  Brazil}
\email{hlopes@im.ufrj.br}

\author[]{Edriss S. Titi}
\address[E.S. Titi]
{Department of Computer Science and Applied Mathematics \\
Weizmann Institute of Science  \\
Rehovot 76100, Israel, and \\
Department of Mathematics \\
and  Department of Mechanical and  Aerospace Engineering \\
University of California \\
Irvine, CA  92697-3875, USA. 
 } \email{etiti@math.uci.edu\\edriss.titi@weizmann.ac.il}
%\email{edriss.titi@weizmann.ac.il}

\begin{abstract}
Helical symmetry is invariance under a one-dimensional group of rigid
motions generated by a simultaneous rotation around a fixed axis and
translation along the same axis. The key parameter in helical symmetry
is the {\it step} or pitch, the magnitude of the translation after rotating one full
turn around the symmetry axis. In this article we study the limits of three-dimensional helical
viscous and inviscid incompressible flows in an infinite circular pipe,
with respectively no-slip  and no-penetration boundary conditions, as
the step approaches infinity. We show that, as the step becomes large,
the three-dimensional  helical flow approaches a planar flow, which is governed by the so-called two-and-half Navier-Stokes and Euler equations, respectively.  
\end{abstract}

\maketitle


%\vskip 0.125in


\vspace{0.5cm}

\today

\vspace{0.5cm}

\section{Introduction}

The helical groups are a family of one-dimensional subgroups of the
rigid motions of three-dimensional Euclidean space consisting of
simultaneous rotation around an axis and translation along the same
axis, for which the ratio of angular rotation to translation is kept
fixed. Each helical group is characterized by a parameter $\sigma \in
\real \setminus \{0\}$, which we call the {\it step} or pitch, defined as the
translation displacement along the symmetry axis after one full {\it
  clockwise} turn around the axis.  The incompressible Navier-Stokes
and Euler equations are covariant under the action of the helical
group. Helically-symmetric or, simply,  ``helical'' flows
 represent a physically interesting class of fluid motions, which
 interpolate between two-dimensional flows and axisymmetric flows,
 see  for instance \cite{CLS99}. Indeed, the helical  groups lie between
 rigid translations in one direction,  associated with 2D flows,  and
 rotation around a fixed axis, associated with axisymmetric
 flows. These regimes 
correspond to formally taking the limits $\sigma \to \infty$ and
$\sigma \to 0$, respectively. The main goal of this work is to examine
the precise nature of the limit $\sigma \to \infty$ for helical
flows, in the case of viscous  and inviscid incompressible flows in a
 circular pipe satisfying,  respectively, no-slip and no-penetration boundary
 conditions.
The limit $\sigma \to 0$ is more technical and, in some sense, less
interesting, as we expect that helical flows will converge  in the limit to
axisymmetric, planar flows, a trivial special case  of axisymmetric flows.  
In fact, periodicity in this case implies asymptotically 
high-frequency oscillations, with weak averaging in the vertical
direction. The analysis of the limit $\sigma\to 0$ is 
closely related  to that in  some of the thin  domain
literature, particularly the special case referred to as {\it PD}, or 
periodic Dirichlet (see \cite{IR01} for more details.) We reserve to
study the limit $\sigma \to 0$ in future work.

We begin by recalling the known mathematical results concerning
helical flows.  As it is the case of two-dimensional flows and axisymmetric flows in
cylindrical domains bounded away from the axis of symmetry,  viscous
incompressible helical flows are globally well posed. This result was proved
by A. Mahalov, E. Titi and S. Leibovich in \cite{MTL}. 
In fact, for the case of a circular pipe they established both global existence of a weak helical solution
with initial data in $L^2$, and global existence and uniqueness of a
strong solution with initial data in 
the Sobolev space $H^1$. (For a discussion about
uniqueness of weak solutions, , within the class of all Leray-Hopf weak solutions of the three-dimensional Navier-Stokes with helical initial data, see
\cite{BLNLT12}.) The situation is different, and rather 
interesting, in the case of ideal fluid governed by the Euler equations, see \cite{D,ET}.
As a matter of fact, an additional  geometric condition is imposed on inviscid
flows, akin to assuming no swirl in the axisymmetric setting, which we
call  {\em no helical swirl} or {\em no helical stretching}.
Under this condition, B. Ettinger and E. Titi \cite{ET} showed global
existence and uniqueness of weak solutions in an appropriate
vorticity-stream function formulation. This formulation can be used,
because, even for finite $\sigma$,  the flow is essentially
two-dimensional, in the sense that it  is
completely determined by the dynamics of the first two components of
the velocity field restricted to any cross section of the pipe.

The main result of this work is a convergence result of  helical flows
to certain flows, the dynamics of which is two dimensional. For this
reason, we will call such limits planar flows, even though the
velocity field can still have three non-zero components. More
precisely, we show that, in the limit $\sigma\to \infty$, helical
flows converge, respectively,  to so-called $2$ and
$1/2$ dimensional flows in the viscous case, and to 2D Euler flows in
the inviscid case. These results are established by first
obtaining a set of symmetry-reduced equations equivalent to the
original  fluid equations, at least for regular flows. 
The unknowns in
these equations are fields on a cross section of the pipe and, hence,
depend on two spatial variables only.
Convergence is then investigated via   energy methods and compactness
arguments. For the Navier-Stokes equations, energy estimates are
sufficient to pass to the limit and give us a rate of convergence of
order $1/\sqrt{\sigma}$ in the energy norm.

One special difficulty in the viscous case  is the way in which the
divergence-free condition and the symmetry reduction interact when we
vary $\sigma$. To be more precise, the symmetry reduction amounts to
the fact that a helical vector field is entirely 
determined by its trace on a horizontal slice, say $D = \{x_1^2 +
x_2^2 < 1, x_3 = 0\}$, the trace being a three-component vector field
in the plane. For a given $\sigma>0$ all three-component
fields in $D$ may be extended in a unique way to  helical vector
fields in $D \times (0,\sigma)$. However, the resulting extension will
not be divergence-free unless the original field in the slice
satisfies a certain $\sigma$-dependent condition. In other words, 
after symmetry reduction, 
problems with different $\sigma$ reside in different function spaces,
even if their physical domain $D$ is the same.
This difficulty is bypassed in the inviscid case with the use of 
a stream function, under the ``no helical swirl'' condition.

The remainder of this article is divided into four sections. In
Section \ref{prelim}, we fix 
notation and derive an equivalent formulation of helical symmetry for
functions and vector fields. In Section \ref{NS}, we perform the
symmetry reduction on the Navier-Stokes equations. In Section \ref{limit}
we study the limit $\sigma \to \infty$ for the viscous case, while in
Section \ref{Euler} we discuss the case of the Euler equations.

\section{Preliminaries and symmetry reduction} \label{prelim}

We begin by recalling some standard notation for function spaces
that will appear throughout the paper. If $\Omega$ is a domain in
$\RR^d$, we denote by $H^k(\Omega)$, $k\in \NN$, the standard
$L^2$-based Sobolev spaces:
\[
     H^k(\Omega)=\{f:\Omega \to \RR\ ; \ f, \pa^\alpha f \in
     L^2(\Omega), |\alpha|\leq k\},
\]
where we employed the usual multiindex notation for derivatives, which
are interpreted in the weak sense, while $W^{k,p}(\Omega)$ denotes $L^p$-based
Sobolev spaces.
By abuse of notation, if $\mbfu:\Omega\to\RR^d$ is a vector field, we will
often write $\mbfu\in H^k(\Omega)$ for $u\in (H^k(\Omega))^d$, and we will
drop the explicit dependence on the domain $\Omega$ when no confusion
can arise.
$H^1_0(\Omega)$ will denote the subspace of $H^1(\Omega)$ of functions
with zero trace at the boundary $\pa\Omega$.
If $\Omega$ is an unbounded domain, $L^p_{\loc}(\Omega)$ is the space
of functions with $p$-th integrable power on each bounded open subset of $\Omega$.
Lastly, we denote H\"older spaces by $C^\alpha(\Omega)$, $\alpha \in \RR_+$. 
Later in the paper, we will introduce other spaces adapted to the
symmetry and 
geometry of the problem.
Throughout, $(,)$ will denote the standard $L^2$ inner product.

One tool that will be used repeatedly in the analysis is 
the following interpolation inequality in two space dimensions, 
 the so-called  {\em Ladyzhenskaya inequality}.\
If $D$ is a smooth domain in $\RR^2$ and $f \in H^1_0(D)$, then 
\begin{equation} \label{ladyineq2D}
\|f\|_{L^4(D)}^4 \leq 2 \|f\|_{L^2(D)}^2\|\nabla f\|_{L^2(D)}^2.
\end{equation} 
This inequality  follows immediately from
Lemma 1 on page  8 of \cite{Ladyzhenskaya1969}.

Let $\Omega = \{x=(x_1,x_2,x_3)\in\real^3 \;|\; x_1^2+x_2^2\leq 1\}=
D\times \RR$
be the infinite pipe with unit circular cross-section $D$ parallel to the
$x_3$-axis.

We consider the initial-boundary-value problem for the incompressible 
Navier-Stokes  (NSE)  and Euler equations (EE) in $\Omega$. 
We recall the notion of helically symmetric
solutions of these equations, studied in \cite{ET,MTL}.

\begin{comment}
\begin{equation} \label{3DNScart}
\left\{
\begin{array}{ll}
\partial_t \mathbf{u} + (\mathbf{u}\cdot\nabla)\mathbf{u} = -\nabla p + \Delta \mathbf{u}, & \mbox{ in }(0,+\infty) \times \Omega; \\
\mbox{div }\mathbf{u} = 0, & \mbox{ in } [0,+\infty) \times \Omega;\\
\mathbf{u} = 0, & \mbox{ on } [0,+\infty) \times \partial\Omega;\\
\mathbf{u}(0,x) = \mathbf{u}_0, & \, x\in\Omega.
\end{array}
\right.
\end{equation}
In this section we discuss helically symmetric solutions of
\eqref{3DNScart}.
\end{comment}

We first  give the definition of a helical vector field and a helical
(scalar)  function. 
We denote a point in $\RR^3$ by $x=(x_1,x_2,x_3)$ in Cartesian coordinates. 
Given a non-zero number $\sigma\in \RR$, we define the action of the
helical group of transformations $G_\sigma$ on $\RR^3$ by:
\[
            S(\rho) (x) = \left (\begin{matrix} x_1\, \cos \rho +  x_2 \sin
              \rho \\  -x_1 \, \sin \rho + x_2 \, \cos \rho \\ x_3 +
              \frac{\sigma}{2\pi} \rho \end{matrix} \right), \qquad
          \rho \in \real,
\]
that is, a rotation around the $x_3$ axis with
simultaneous translation along the $x_3$ axis.
$G_\sigma$ is uniquely determined by $\sigma$, 
which we will call  the {\em step}  (or pitch). Invariant curves for the action of
the helical group $G_\sigma$ are helices  having the $x_3$ axis as axis of
symmetry.  The cylinder $\Omega$ is an
invariant set for the action of $G_\sigma$ for all $\sigma$.
A change of sign in $\sigma$ corresponds to switching the orientation
of the helices preserved by the group action from right-handed to
left-handed. Without loss of generality, we will restrict our
attention to the case of $\sigma >0$.

We will say that the smooth function $f(x)$ is {\em helically
  symmetric}, or simply {\em helical},   if  $f$
is invariant under the action of $G_\sigma$, i.e., \ 
$f(S_\rho x) = f(x)$, $\forall \rho \in \real$. Similarly, we say that
the smooth vector field $\mbfu(x)$ is helically symmetric, or simply
helical, if  it is covariant with
respect to the action of $G_\sigma$, i.e., \ $M(\rho) \mbfu(x)  =
\mbfu(S(\rho) x)$ for all $\rho \in \real$, where 
\begin{equation} \label{Mofx3}
M(\rho) : =
\left[\begin{array}{ccc}
\cos \rho & \sin\rho & 0 \\ \\
-\sin\rho & \cos\rho & 0 \\ \\
0 & 0 & 1
\end{array}\right].
\end{equation}

%Let $(r,\theta,z)$ denote standard cylindrical coordinates in $\RR^3$
%with Cartesian coordinates , i.e.,
%$x_1 = r \cos \theta$, $x_2 = r \sin \theta$, $x_3 = z$. 
%The orthonormal frame associated to these coordinates is then 
%\[
%\mathbf{e_r}=(\cos\theta,\sin\theta,0), \;\;\; \mathbf{e_{\theta}}=(-\sin\theta,\cos\theta,0), \;\;\;
%\mathbf{e_z}=(0,0,1). 
%\]

We find it convenient to give an alternative definition of helical
symmetry as follows.
We re-write a vector field
$\mathbf{u}(x)=(u^1,u^2,u^3)(x_1,x_2,x_3)$ with respect to the moving
orthonormal frame associated to standard cylindrical 
coordinates  $(r,\theta,z)$,
\[
  \mathbf{e_r}=(\cos\theta,\sin\theta,0), \;\;\; \mathbf{e_{\theta}}=(-\sin\theta,\cos\theta,0), \;\;\;
 \mathbf{e_z}=(0,0,1),
\]
as:
\[
  \mathbf{u} = u_r \mathbf{e_r} + u_{\theta} \mathbf{e_{\theta}} + u_z
  \mathbf{e_z},
\]
where $u_r$, $u_{\theta}$, $u_z$ are functions of $(r,\theta,z)$.
We introduce two new independent variables in place of $\theta$ and
$z$:
\begin{align}\label{var}
\eta: =\frac{\sigma}{2\pi} \theta+ z,\ \xi : = \frac{\sigma}{2\pi}
\theta - z.
\end{align}
As shown in \cite{ET} for instance, a
(smooth) function $p=p(r,\theta,z)$  is a helical function 
if and only if, when expressed in the $(r,\xi,\eta)$ variables, it is
independent of $\xi$:\  
$p=q(r,\frac{\sigma}{2\pi}\theta+ z)$, for some $q=q(r,\eta)$
%A vector field is said to respect the helical symmetry if its
%components in the moving  frame $(\mathbf{e_r}, \mathbf{e_{\theta}},\mathbf{e_z})$
%are helically symmetric, see \cite{ET,MTL} for details and properties
%of helical vector fields.
%Hereafter, we assume, without loss of generality, that helically
%symmetric functions, written in the variables $(r,\xi,\eta)$, are
%independent of $\eta$.
Similarly,  a (smooth) vector field $\mathbf{u}$ is helical if and only if there exist
$v_r$, $v_{\theta}$, $v_z$,
functions of $(r,\eta)$ such that $u_r = v_r(r,\frac{\sigma}{2\pi} \theta + 
z)$, $u_{\theta} = v_{\theta}(r,\frac{\sigma}{2\pi}\theta +  z)$,
$u_z = v_z(r, \frac{\sigma}{2\pi} \theta + z)$. 

We note that a vector field $\mbfu$ is invariant  under the action of
$G_\sigma$ for all $\sigma\ne 0$ if and only if  $v_r$, $v_{\theta}$,
$v_z$ are  functions of $r$ only.
In particular, planar, circularly symmetric flows, that is flows for which $v_r=v_z\equiv 0$ and $v_\theta$ is a radial function, are a (very) special case of helical flows.

The change of variables $(x_1,x_2,x_3) \mapsto (r,\xi,\eta)$
introduced 
above
 has often been used to characterize helical symmetry, and, in fact,
 it does 
provide a simple, geometrically elegant description of invariance for
both 
scalar functions and vector fields. However, to obtain estimates on 
solutions of the fluid equations, we find that an alternative 
characterization actually simplifies calculations, by avoiding moving 
frames. 
As a matter of fact, we show in the following proposition that sufficiently smooth functions and
fields with helical symmetry are essentially two dimensional, in the
sense that they are uniquely determined by their trace on any
``slice'' $\Omega \cap  \{z=\text{constant}\}$, which can be
canonically identified
with the unit disk $D\subset \RR^2$.

Below we will make use of the following notation, where we employ
Cartesian coordinates and frames.
Given  $y=(y_1,y_2)$ we let $y^{\perp} = (-y_2,y_1)$ and we set 
\begin{equation} \label{E1} 
E \equiv y^{\perp}\cdot\nabla_y. 
\end{equation} 
We also use the notation $\mathbf{V}_H=(V^1,V^2,0)$ for the 
horizontal component of the vector $\mathbf{V}=(V^1,V^2,V^3)$, 
 and we  denote the vector  $(- V^2,V^1,0)$  by
 $\mathbf{V}_H^{\perp}$.

\begin{proposition} \label{HelicalCartCorr}
Let $\mathbf{u} = \mathbf{u}(x) $ be a smooth helical vector field and
let $p=p(x)$ be a smooth helical function, where
$x=(x_1,x_2,x_3)$. Then there exist {\em unique\/} $\mathbf{w}=(w^1,w^2,w^3)
=(w^1,w^2,w^3)(y_1,y_2)$ and $q=q(y_1,y_2)$ such that
\begin{equation} \label{uandpinY}
\mathbf{u}(x) = M(2\pi x_3/\sigma)\mathbf{w}(y(x)), \;\;\; p=p(x)=q(y(x)),\end{equation}
with $M(\rho)$ given in \eqref{Mofx3},
%\begin{equation} \label{Mofx3}
%M(\rho) =
%\left[\begin{array}{ccc}
%\cos \rho & \sin\rho & 0 \\ \\
%-\sin\rho & \cos\rho & 0 \\ \\
%0 & 0 & 1
%\end{array}\right],
%\end{equation}
and
\begin{equation} \label{Yofx}
y(x) =
\left[\begin{array}{l}
y_1 \\ \\
y_2
\end{array}\right]
=
\left[\begin{array}{cc}
\cos(2\pi x_3/\sigma) & -\sin(2\pi x_3/\sigma) \\ \\
\sin(2\pi x_3/\sigma) & \cos(2\pi x_3/\sigma)
\end{array}\right]
\left[\begin{array}{l}
x_1\\ \\
x_2
\end{array}\right].
\end{equation}

Conversely, if $\mathbf{u}$ and $p$ are defined through \eqref{uandpinY} for some 
$\mathbf{w}=\mathbf{w}(y_1,y_2)$, $q=q(y_1,y_2)$,
then $\mathbf{u}$ is a helical vector field and $p$ is a helical scalar function.
\end{proposition}

We omit the proof, which is a standard application of vector calculus.

In what follows, for notational convenience we set
\[m^{\sigma}(x_3) = \left[\begin{array}{cc}
\cos(2\pi x_3/\sigma) & -\sin(2\pi x_3/\sigma) \\ \\
\sin(2\pi x_3/\sigma) & \cos(2\pi x_3/\sigma)
\end{array}\right],\]
so that 
\[y(x_1,x_2,x_3) = \left[\begin{array}{c} y_1\\ y_2 \end{array} \right] = m^{\sigma}(x_3)\left[\begin{array}{c} x_1\\ x_2 \end{array} \right],\]
and
\[M^{\sigma} (x_3) \equiv M(2\pi x_3/\sigma)= \left[\begin{array}{cc}
(m^{\sigma}(x_3))^T & 0 \\ \\
0 & 1
\end{array}\right].\]

It is clear, from Proposition  \eqref{HelicalCartCorr} above that any smooth helical flow is
periodic in $x_3$, both velocity and pressure, with period the step
$\sigma$.
We can therefore state the initial-boundary-value problem for 
 the Navier-Stokes equations in the fundamental
domain \  $\Omega^\sigma : = D \times (0,\sigma)$:
\begin{equation} \label{3DNSper}
\left\{
\begin{array}{ll}
\partial_t \mathbf{u} + (\mathbf{u}\cdot\nabla)\mathbf{u} = -\nabla p
+ \nu \Delta \mathbf{u} +\mathbf{f}, & \mbox{ in }(0,+\infty) \times \Omega^\sigma; \\
\dive\, \mathbf{u} = 0, & \mbox{ in } [0,+\infty) \times
\Omega^\sigma;\\
\mathbf{u}(t,x',x_3) = 0, & \mbox{ for  } t \in [0,+\infty), 
\quad |x'|=1, \; 0\leq x_3\leq \sigma;\\
\mathbf{u}(t,x',x_3) =\mbfu(t,x',x_3+\sigma) & \mbox{ for  } t \in [0,+\infty), 
\quad x'\in D;\\
p(t,x',x_3) =p(t,x',x_3+\sigma) & \mbox{ for  } t \in [0,+\infty), 
\quad x'\in D;\\
\mathbf{u}(0,x) = \mathbf{u}_0, & \, x\in\Omega^\sigma,
\end{array}
\right.
\end{equation}
where we set $x'=(x_1,x_2)$, so that $x= (x',x_3)$. 

The Euler equations are formally obtained by setting $\nu=0$ above and
by replacing the no-slip boundary condition $\mbfu\vert_{\pa\Omega^\sigma} =0$
with the no-penetration condition $\mbfu\cdot x' =0$ on $\pa\Omega^\sigma$.
We discuss Euler solutions in Section \ref{Euler}.

In what follows, for simplicity we set any body
forcing $\mathbf{f}\equiv 0$, and take the viscosity coefficient $\nu =1$, as
we do not contend ourselves with the vanishing viscosity limit in this
work.  We  plan to study  the interplay between the limits $\nu \to 0$ and $\sigma\to
\infty$  in future work.

We denote by $C^\alpha_{per}(\overline{\Omega^\sigma})$ the subspace
of  $C^\alpha(\overline{\Omega})$, $\alpha\in \RR_+$ of functions that are
$\sigma$-periodic in $x_3$, and by $C^{\infty}_{c,per}(\Omega^{\sigma})$ the space of functions which are $\sigma$-periodic in $x_3$  and  compactly supported in $D$ for 
each fixed $x_3 \in [0,\sigma]$.  We  also denote by
$H^1_{0,per}(\Omega^{\sigma})$ the closure of
$C^{\infty}_{c,per}(\Omega^{\sigma})$ in $H^1(\Omega^{\sigma})$, 
and
by $H^{-1}_{per}(\Omega^\sigma)$ its dual. We note
that the closure of the subspace of
$C^{\infty}_{c,per}(\Omega^{\sigma})$ of divergence-free vector field
is  the subspace   \ $\{ \mbfu \in H^1_{0,per}(\Omega^{\sigma}) \; | \;
\dive \, u =0\}$, where derivatives are taken in the weak sense.

In the remainder of the paper we will consider solutions to
\eqref{3DNSper} and the corresponding inviscid system \eqref{3DEEper}  with initial data $u_0$ of
limited regularity. More specifically, $u_0$  will be taken 
 in $H^1_{0,per}(\Omega^{\sigma})$ for Navier-Stokes and in
 $H^1_{per}(\Omega^\sigma)$ with initial vorticity $\curl \, u_0\in
 L^\infty(\Omega^\sigma)$ for Euler. We now briefly discuss helical
 symmetry in this context.

\begin{defn} \label{weakhelical}
Let $p \in H^1_{per}(\Omega^\sigma)$. We say that $p$ has helical
symmetry if there exists a sequence of smooth, helical functions $p_n$
such that
$p=\lim_{n\to \infty} p_n$ in   $ H^1_{per}(\Omega^\sigma)$. Similarly,
we say that a vector field $\mbfu$ in
$H^{1}_{per}(\Omega^{\sigma})^3$ has helical symmetry if $\mbfu$ is a
strong limit in $H^{1}_{per}(\Omega^{\sigma})^3$ of a sequence of
smooth, helical vector fields $\mbfu_n$.
\end{defn}

We next show that the characterization of helical symmetry given in
Proposition \ref{HelicalCartCorr} carries over to functions and vector
fields in $H^1$. 

\begin{proposition} \label{HelicalCartCorweak}
Let $\mbfu\in (H^1_{per}(\Omega^\sigma))^3$, $p\in
H^1_{per}(\Omega^\sigma)$ 
be, respectively,  a helical vector field and a helical
function.  Then, there exist a unique
$\mw \in H^1(D)^3$ and $q\in H^1(D)$, where $D$ is the unit disk in
$\RR^2$, such that 
\begin{equation} \label{wdefweak}
    \begin{aligned}
        &\mbfu(x) = M^\sigma(x_3) \mw(m^\sigma(x_3) x'),\\
       &p(x) = q(m^\sigma(x_3) x')),
  \end{aligned}      
 \qquad \text{a. e. }  x'\in D, \quad \forall \; 0\leq x_3\leq \sigma.
\end{equation}
Conversely, given $\mw \in H^1(D)^3$ and $q\in H^1(D)$, if $\mbfu$ and
$p$ are defined through  \eqref{wdefweak}, then $\mbfu \in
(H^1_{per}(\Omega^\sigma))^3$, $p \in (H^1_{per}(\Omega^\sigma))$, and
they have helical symmetry.
\end{proposition}

\begin{proof}
We only consider the case of a helical vector field $\mbfu$. The case
of a helical function is similar and simpler.
By definition, there exist helical vector fields $\mbfu_n \in
C^\infty(\Bar\Omega^\sigma)$ such that $u_n \to u$ strongly in
$H^1_{per}(\Omega^\sigma)$. By Proposition \eqref{HelicalCartCorr}, for each
$\mbfu_n$ there exists a unique, smooth $\mw_n$ such that \ 
$  \mw_n(x')= (M^\sigma(x_3))^T \mbfu_n((m^\sigma(x_3))^T x')$,
for all $x'=(x_1,x_2)\in D$. Therefore, the expression on the
right-hand side is independent of $x_3$ and 
 \ $ \nabla_x (M^\sigma(x_3))^T \mbfu((m^\sigma(x_3))^T x') =
 (\nabla_{x'} \mw(x'),0)$.
If we define 
\[
 \mw(x',x_3) :=  (M^\sigma(x_3))^T
\mbfu((m^\sigma(x_3))^T x'),
\]
then $\partial_{x_3} \mw(x',x_3)=0$ in weak sense,
since it is true for $\mw_n$ and $\nabla_x\mbfu_n\to \nabla \mbfu$
strongly in $L^2(\Omega^\sigma)$. Consequently, $\mw$ is independent of $x_3$ for
almost all $x'\in D$ (functions with vanishing weak derivatives are
constant, see e.g. \cite[Theorem 6.11]{LiebLoss}) and $\mw\in
H^1(D)$. Furthermore, 
\[
\|\mw_n -\mw\|_{H^1(D)}\leq C \sqrt{\sigma} \, \|\mbfu_n-\mbfu\|_{H^1(\Omega_\sigma)},
\]
by a simple change of variables, so that $\mw_n\to \mw$ strongly in
$H^1(D)$.
The converse statement is a direct consequence of \eqref{wdefweak}.
\end{proof}

\begin{rem} \label{helicalSobolev}
The proof of Proposition \ref{HelicalCartCorweak} shows that if $\mbfu\in
H^m(\Omega^\sigma)$, $m\in \NN$, then $\mw\in H^m(D)$ and  
the $H^m$ norm  of $\mw$ on $D$ is bounded by the $H^m$ norm of
$\mbfu$ on $\Omega^\sigma$ with
constants that depend on $\sigma$. The same result holds in
$L^p$-Sobolev spaces $W^{m,p}_{per}(\Omega^\sigma)$ for  $1\leq p
<\infty$. These spaces are defined in a manner totally analogous to 
$H^m_{per}(\Omega^\sigma)$.
\end{rem}

We next recall the notion of weak and strong Navier-Stokes solutions. 
By a classical solution of  \eqref{3DNSper} on the time interval
$[0,T]$,  we mean a vector field  $\mbfu \in
C^1([0,T];C^2(\overline{\Omega^\sigma})$, together with a function $p
\in C^1([0,T),C^1(\Omega_\sigma))$ such  that the equations, and 
the initial and boundary conditions are met pointwise in $t$ and $x$.
By a weak solution  on the time interval $[0,T)$,
 we mean a divergence-free vector field $\mbfu:[0,T) \times
 \Omega^\sigma \to \RR^3$ 
such  that  \ $\mbfu\in C_{w}([0,T);L^2(\Omega^\sigma))\cap
 L^2((0,T);H^1_{0,per}(\Omega^\sigma))$  and 
$\pa_t \mbfu \in L^1((0,T),H^{-1}_{per}(\Omega^\sigma))$, satisfying the
equations in the sense of distributions and the initial condition 
$\mbfu(0) = \mbfu_0\in L^2(\Omega^\sigma)$. Here, $C_w([0,T);L^2)$ is
the space of all functions of $t$ with values in $L^2$ that are
continuous w.r.t. the weak topology on $L^2$. 
We remark that weak solutions satisfy the 
Dirichlet (no-slip) boundary conditions at least in trace sense 
on the boundary for almost all $0<t<T$.
By a strong  solution we mean a weak solution  that satisfies in
addition \ $u \in L^\infty([0,T);H^1_{0,per}(\Omega^\sigma))\cap 
L^2((0,T);H^2_{per}(\Omega^\sigma)\cap H^1_{0,per}(\Omega^\sigma))$ and the condition 
$\mbfu_0\in H^1_{0,per}(\Omega^\sigma).$
It then follows that there exists an associated pressure function 
$p\in L^2((0,T); H^1(\Omega^\sigma))$.  
A strong helical solution will denote a strong solution that is a
helical field in the sense of Definition \ref{weakhelical}.
We recall that any strong solution of the Navier-Stokes equations is
unique and  smooth for $t>0$ (see e.g., \cite[Theorem
1.8.2]{Sohr}). Hence, strong solutions are actually classical
solutions on any time interval $[\delta,T]$, $\delta>0$.
It was shown in 
\cite[Theorem 3.4]{MTL} that weak solutions of \eqref{3DNSper} with
helical  symmetry  are unique, global in time, and agree with a strong
solution, if the initial data belongs to 
$H^1_{0,per}(\Omega^\sigma)$ and the associated pressure $p$ is also a helical function.
(See also \cite{BLNLT12} for more elaborate discussion regarding this matter.)

\section{Symmetry reduction for the Navier-Stokes equations} \label{NS}

In this section we derive a set of symmetry-reduced equations that
completely capture the dynamics of the original system under the
hypothesis of helical symmetry.

We begin by deriving the symmetry-reduced system under the hypothesis
that  $(\mbfu, p)$ are  classical solutions of \eqref{3DNSper} 
and have helical symmetry.
Let $\mathbf{w}=\mathbf{w}(t,y_1,y_2)$ be given in terms of $\mbfu$ by
Proposition \ref{uandpinY}. We will
 derive from Navier-Stokes the equations satisfied by $\mw$.
 Smoothness of $\mbfu$ and $\mw$ justifies all the algebraic
 manipulations.
For ease of notation, in this proof we write $M^T$ for  $[(M^\sigma)(x_3)]^T$.
We multiply the momentum  equation in \eqref{3DNSper}
by  $M^T$  and identify each term in the resulting
expression as follows to obtain:
\begin{subequations}
\begin{align} 
  &M^T \partial_t \mathbf{u} = \partial_t \mathbf{w}, \label{partialt} \\
 &M^T[(\mathbf{u}\cdot\nabla_x)\mathbf{u}] = (\mathbf{w}_H \cdot \nabla_y)\mathbf{w}
+ (\frac{2\pi}{\sigma})w^3  E\mathbf{w}
- (\frac{2\pi}{\sigma})w^3  \mathbf{w}_H^{\perp}, \label{nonlinear} \\
&M^T\nabla_x p = (\nabla_y q)_H + (\frac{2\pi}{\sigma})
Eq\mathbf{e_3}, \label{nablap} \\
&M^T\Delta_x\mathbf{u} = \Delta_y\mathbf{w} + (\frac{2\pi^2}{\sigma^2})[E^2\mathbf{w} -
2E\mathbf{w}_H^{\perp} - \mathbf{w}_H], \label{Laplacian}
\end{align}
\end{subequations}
where $E$ is the operator defined in \eqref{E1}.
We similarly  perform the symmetry reduction on the incompressibility
condition for $\mbfu$ to obtain
\begin{equation} \label{divfree}
 \mathrm{div}_x \mathbf{u} = \mathrm{div}_y \mathbf{w}_H +
 (\frac{2\pi}{\sigma}) Ew^3.
\end{equation}
Therefore, we find that $\mw$ and $q$ satisfy the following
initial-boundary-value problem:
%\begin{minipage}
\begin{subequations} \label{3DNStwistedcart}
   \begin{align}&\partial_t \mathbf{w} + (\mathbf{w}_H\cdot\nabla_y)\mathbf{w} + 
\ds{\frac{2\pi}{\sigma}}w^3[E\mathbf{w} - \mathbf{w}_H^{\perp}]
= -(\nabla_y q)_H 
\nonumber  \\
    &\qquad \qquad   -\ds{\frac{2\pi}{\sigma}}Eq\mathbf{e_3} 
  +\Delta_y \mathbf{w} + \ds{\frac{4\pi^2}{\sigma^2}} [E^2\mathbf{w}
- 2 E \mathbf{w}_H^{\perp} - \mathbf{w}_H],  \label{3DNStwistedcart1}\\  
&\mathrm{div}_y \mathbf{w}_H + \ds{\frac{2\pi}{\sigma}} E w^3 = 0,
\qquad \qquad \qquad t>0, \quad y \in D, \label{3DNStwistedcart2}\\
     &\mw(t,y) = 0,   \qquad \qquad \qquad \qquad \qquad  \quad t>0, \quad  |y|=1, \\
     & \mw(0,y) = \mw_0(y),  \qquad \quad \qquad \qquad \qquad y \in D,
   \end{align}
\end{subequations}
%\end{minipage}
where $\mw_0$ is related to $\mbfu_0$ via \eqref{uandpinY}.

Before giving a weak formulation of the above initial-boundary-value
problem, we note that the operator $E=y^\perp\cdot \nabla_y$  is
 anti-selfadjoint, i.e., $E^\ast =  -E$, since $\dive_y \, y^\perp=0$.
If  we write \eqref{3DNStwistedcart2} as  \ $A\, \mw=0$, 
for some matrix operator $A$ with $\mw$ a column vector,  it
follows that $A$ and its adjoint $A^\ast$ are given by:
\[
      A := \left [\begin{matrix} \pa_{y_1}, & \pa_{y_2},  &
          \frac{2\pi}{\sigma} E \end{matrix} \right],
     \qquad A^\ast :=  \left[ \begin{matrix}  - \pa_{y_1} \\ -
         \pa_{y_2} \\  -\frac{2\pi}{\sigma} E
     \end{matrix} \right].
\]
It can be easily checked that  the (scalar) second-order operator \
$A\,A^\ast = - \Delta_y -
\frac{4\pi^2}{\sigma^2} E^2 $ \ is elliptic for any $\sigma\ne 0$.

 We will call  a vector field $\mw$ on $[0,T)\times D$ 
a weak solution of \eqref{3DNStwistedcart} if  $\mw \in C_{w}([0,T);L^2(D))\cap
 L^2((0,T);H^1_0(D))$,   
$\pa_t \mw \in L^1((0,T);H^{-1}(D)$,  $\mw(0)=\mw_0 \in L^2(D)$, $\mw$
satisfies the  constraint
\eqref{3DNStwistedcart2} in th sense of distributions, and for all
(vector-valued) test functions
$\mPhi \in C^\infty_c([0,T)\times D)$  that satisfy \eqref{3DNStwistedcart2}, 
\begin{multline} \label{3DNStwistedweak}
    \int_0^t \int_D \mw\cdot \pa_t \mPhi \, dy\, dt  +
    \frac{2\pi}{\sigma} \int_0^t \int_D w^3 ( \mPhi \cdot \mw^\perp_H +
    E \mPhi \cdot \mw)\,
    dy \, dt +   \\
    \int_0^t \int_D   \Delta \mPhi \cdot
    \mw\, dy\, dt  + \frac{4\pi^2}{\sigma^2}  \int_0^t \int_D \left ( E^2 \mPhi \cdot \mw +
    2 E\mPhi \cdot \mw_H^\perp \right) \, dy\, dt \\
     - \frac{4\pi^2}{\sigma^2}  \int_0^t \int_D \mPhi\cdot \mw_H \,
     dy\, dt= \int_D \mPhi(0) \cdot \mw(0)\, dy.
\end{multline}
A weak solution will be called a strong solution if, in addition, 
$\mw \in L^\infty([0,T);H^1_0(D))$
$\cap  L^2((0,T);H^2(D)\cap H^1_0(D))$ and 
$\mbfu_0\in H^1_0(D).$  By interpolation then, $\mw\in
C((0,T);H^1_0(D))$ (c.f. e.g. \cite[Lemma 4.8 p. 570]{TayIII}).
By projecting the momentum equation \eqref{3DNStwistedcart1} onto the
kernel of the operator $A$, one obtains an elliptic equation for the
pressure $q$: 
\begin{equation} \label{qequation}
           A\,A^\ast q  =  A\, \left[  (\mathbf{w}_H\cdot\nabla_y)\mathbf{w} - 
\frac{2\pi}{\sigma} w^3\, (E\mathbf{w} - \mathbf{w}_H^{\perp})
 - \frac{4\pi^2}{\sigma^2} 
(2 E \mathbf{w}_H^{\perp} + \mathbf{w}_H) \right], 
\end{equation}
 and  by elliptic regularity, it follows that  $q \in L^1([0,T); H^1(D))$.

In the following proposition we establish the relationship between
strong solutions to the Naviers-Stokes system \eqref{3DNSper} and
strong solutions of the symmetry-reduced system \eqref{3DNStwistedcart}.

\begin{proposition} \label{helicalYvariables}
Let $\mbfu_0\in H^1_{0,per}(\Omega^\sigma)$ be a divergence-free, helical vector field. 
Let  $\mbfu$ be the unique, strong helical solution of \eqref{3DNSper}
on $[0,T)$, for any $T>0$,
 with initial condition $\mbfu_0$ and associated pressure function $p$.
 Then,  the vector function $\mathbf{w}=
(w^1,w^2,w^3)$ and scalar function  $q$, defined through
\eqref{wdefweak} from $\mbfu$ and $p$, give a strong solution of the reduced system
\eqref{3DNStwistedcart}.

Conversely,  let $\mw$ be a strong  solution of
\eqref{3DNStwistedcart} 
and associated pressure $q$. Then, if
$\mbfu$ and $p$ are defined   from $\mw$ and $q$ via \eqref{wdefweak},
$\mbfu$ is a strong helical solution of \eqref{3DNSper}. In
particular, strong solutions of \eqref{3DNStwistedcart}  are unique.
\end{proposition}

\begin{proof}
By Definition \ref{weakhelical}, there exists a sequence of smooth,
helical functions $\mbfu_{0,n}$ on $\Omega^\sigma$ such that 
\ $\mbfu_{0,n} \to \mbfu_0$ strongly in $H^1_{0,per}(\Omega^\sigma)$.
Let $\mbfu_n$ be the unique, classical  helical solution of
\eqref{3DNStwistedcart} with initial data $\mbfu_{0,n}$, and pressure $p_n$. The sequence 
$\{\mbfu_n\}$ is uniformly bounded in $L^\infty([0,T);H^1_{0,per}(\Omega^\sigma))\cap 
L^2((0,T);H^2_{per}(\Omega^\sigma)\cap H^1_{0,per}(\Omega^\sigma))$
and $\{\pa_t \mbfu_n\}$ is uniformly bounded in $L^1([0,T);H^{-1}_{per}(\Omega^\sigma))$.
Therefore, by interpolation and Rellich's theorem,  there exists a subsequence converging strongly in
$H^{-\epsilon}([0,T);H^1_{0,per}(\Omega^\sigma))\cap L^2((0,T);H^{1-\epsilon}_{per}(\Omega^\sigma))$, for all $\epsilon>0$, 
weakly in $L^2((0,T);H^2(\Omega^\sigma))$, and weakly-$\ast$ in 
$L^\infty([0,T);H^1_{0,per}(\Omega^\sigma))$,
such that $\pa_t \mbfu_n$ converges weakly 
 in $L^1((0,T);H^{-1}(D))$. The limit  $\mbfu$ is then a weak solution  of \eqref{3DNStwistedcart} with
initial data  $\mbfu_0$ (by arguments similar to those showing
existence of Leray-Hopf weak solutions, cf. \cite[Theorem 5.9,
Chap. 17]{TayIII}. )

Since weak solutions agree with strong
solutions as long as the latter exists, we must have that $\mbfu$ is
the unique, strong helical solution of \eqref{3DNStwistedcart} with
initial data $\mbfu_0$. Hence, the whole sequence $\{\mbfu_m\}$
converges to $\mbfu$ by uniqueness of  the
limit. A similar argument gives convergence of $p_n$ to $p$ in
$L^1((0,T);H^1(D))$ .

Let now $\mw_n$ be associated to $\mbfu_n$ by \eqref{uandpinY}. Then,
$\mw_n$ is a classical solution of \eqref{3DNStwistedcart}, with
associated pressure $q_n$ given by \eqref{uandpinY} in terms of $p_n$,
by the
calculations at the beginning of this section. Furthermore,  the proof
of Proposition \ref{weakhelical} implies that all Sobolev norms of $\mw_n$
and $q_n$ are bounded by
the corresponding Sobolev norms of $\mbfu_n$ with constants depending
on $\sigma$. Hence, the sequence $\{\mw_n\}$ is uniformly bounded in 
$L^\infty([0,T);H^1_{0}(D))\cap 
L^2((0,T);H^2(D)\cap H^1_{0}(D))$.  From the equations, it follows that $\pa_t \mw_n$ is uniformly
bounded in $L^1((0,T);H^{-1}(D))$.
Hence, by interpolation and Rellich's theorem there exists a
subsequence converging 
strongly in $H^{-\epsilon}([0,T);H^1_0(D))\cap
L^2((0,T);H^{1-\epsilon}(D))$, for all 
$\epsilon>0$,  weakly in   
$L^2((0,T);H^2(D))$, and weakly-$\ast$ in $L^\infty([0,T);H^1_0(D))$ to a
weak solution $\mw$ of the symmetry-reduced system
\eqref{3DNStwistedcart}.
Since $\mw \in L^\infty([0,T);H^1_{0}(D))\cap 
L^2((0,T);H^2(D)\cap H^1_{0}(D))$,
$\mw$ is a strong solution of the reduced system.  Also, by refining
the subsequence if needed, we can assume that $\{q_n\}$ converges
weakly in  $L^1((0,T);H^1(D))$. Furthermore, the convergence of
$\mbfu_n$ to $\mbfu$ implies weak convergence of the right-hand side
of \eqref{qequation} in $L^1(0,T);H^{-1}(D))$ and, hence, $q$ is a
weak solution of the pressure equation. Lastly, since  $\mw$
and $q$ in \eqref{wdefweak} are unique,  given $\mbfu$ and $p$, these
must agree with the limits of $\mbfu_n$ and $p_n$. The first half of
the theorem is established.

The converse follows by similar arguments, using again the
uniqueness in the relation between $\mbfu$, $p$ with $\mw$, $q$ of
Proposition \ref{HelicalCartCorweak}. Energy estimates for strong solutions
of the symmetry-reduced equations are given in Propositions
\ref{EnergyTwistedEqsPrime} and \ref{EnergyTwistedEqs}. 
Uniqueness of strong solutions to the
reduced equations then follows from uniqueness of helical, strong
solutions of the Navier-Stokes equations.
\end{proof}

\section{The limit $\sigma \to \infty$ for the Navier-Stokes system} \label{limit}

The purpose of this section is to discuss the limit $\sigma \to
\infty$ for helical solutions of the Navier-Stokes equations. To emphasize the dependence of the solution on the parameter $\sigma$, we will write 
$\mbfu^\sigma$ and $p^\sigma$ for $\mbfu$ and $p$.

Next, we recall that to any helical vector field $\mbfu^\sigma$ in
$H^1(\Omega^\sigma)$ we can associate a three-component vector
function $\mw^\sigma$ in $H^1(D)$ by means of  Proposition \ref{HelicalCartCorweak}.
The divergence-free condition on $\mathbf{u}^\sigma$ is recast as
\eqref{3DNStwistedcart2}  for $\mathbf{w}^\sigma$.  
In what follows, we will need to relate divergence-free vector fields in $D$ 
to fields satisfying  the condition in \eqref{3DNStwistedcart2}.
To this end, we will exploit the following useful lemma.

\begin{lemma} \label{GaldisLemma}
There exists a constant $C>0$ such that, for every $f \in L^2(D)$ with $\int_D f(x) \, dx = 0$, there exists a vector field $\mathbf{v} \in H^1_0(D)$ satisfying
\[\mathrm{div}_y \mathbf{v} = f\;\;\;\mbox{and}\]
\[\|\nabla \mathbf{v}\|_{L^2(D)} \leq C \|f\|_{L^2(D)}.\]
\end{lemma}

\begin{proof}
Since $D$ is clearly star-shaped, this is a special case of Lemma
III.3.1 on page 116 of \cite{G1}.
\end{proof}

We note that $\mbfv$ is not uniquely determined. In fact, we can add
to $\mbfv$ any 
divergence-free vector field in $D$, satisfying the $H^1$ bound above.
The vector field $\mbfv$ can be made unique by assuming, for example,
that it is curl free.

Next, we will state and prove several energy-type estimates for
$\mw^\sigma$. These follow  
from corresponding bounds for $\mbfu^\sigma$ thanks to Proposition  
\ref{helicalYvariables}, but we derive them here keeping track of the
precise dependence on 
the parameter $\sigma$.

Given a helical vector field $\mbfu_0 \in H^1_{0,per}(\Omega^\sigma)$, 
Proposition \ref{helicalYvariables} gives a one-to-one correspondence between strong helical solutions of \eqref{3DNSper} and strong solutions
of \eqref{3DNStwistedcart} with initial data $\mw_0\in H^1_0(D)$ satisfying
\begin{equation} \label{TwistedDivFreeCond}
\mbox{div}_y \left[ (\mathbf{w}_0^{\sigma})_H\right] + \frac{2\pi}{\sigma}E[(w^{\sigma,3}_0)] = 0,
\end{equation}
where $w_0^{\sigma,3}$ refers to the third component of $\mathbf{w}_0^{\sigma}$, and
$\mbfu_0$ and $\mw_0$ are related via \eqref{wdefweak}.
In particular, $\mw \in C([0,T),H^1_0(D))$.

We remark that for any helical vector field 
$\mbfu^\sigma_0$ for which the component along the axis of the pipe,
$u^{\sigma,3}_0$, is a 
radial function, the symmetry-reduced constraint on the divergence is
in fact simply the 
divergence-free constraint in 2D for $(\mw^\sigma_0)_H$, since in this case
$E\,w^{\sigma,3}_0\equiv 0$. In this special case, the analysis  is 
considerably simplified.
We may now state our next results, consisting of energy estimates for 
$\mathbf{w}^{\sigma}$. We split these into two propositions, the first
valid for all  $\sigma > 0$ and the second valid for large $\sigma$.

\begin{proposition} \label{EnergyTwistedEqsPrime}
Given $\sigma>0$, let $\mw^\sigma$ be a strong solution of
\eqref{3DNStwistedcart} on the time interval $[0,T)$.
%with initial condition $\mw^\sigma_0 \in H^1_0(D)$. 
Then, for all $t\in (0,T)$, we have that
\begin{equation} \label{eq.energytwisted}
\begin{aligned}
 \int_D |\mathbf{w}_0^{\sigma}&(y)|^2\,dy=
 \int_D |\mathbf{w}^{\sigma}(t,y)|^2\,dy + 2\int_0^t\int_D|\nabla \mathbf{w}^{\sigma}(s,y)|^2\,dy\,ds  \\ 
& +2\int_0^t\int_D \frac{4\pi^2}{\sigma^2}
\left[(E (w^{\sigma,3})  )^2 + |E\mathbf{w}_H^{\sigma} - (\mathbf{w}_H^{\sigma})^{\perp}|^2 \right] \,dy\,ds. 
\end{aligned}
\end{equation}
\end{proposition}

\begin{proof}
We simply observe that $\mathbf{w}^{\sigma}$ has enough regularity to be a test function in the weak formulation of 
\eqref{3DNStwistedcart}, so we are justified in  multiplying \eqref{3DNStwistedcart} by  $\mathbf{w}^{\sigma}$ and integrating over the domain $D$ and, subsequently, in time. This easily yields the desired identity.
\end{proof}

\begin{proposition} \label{EnergyTwistedEqs}
Let $1 \leq \sigma < \infty,$ and fix $T>0$.
Let $\mbfu^\sigma$ be a strong helical  solutions of \eqref{3DNSper}
on the interval $[0,T)$.
%, with initial data $\mbfu_0\in H^1_{0,per}(\Omega^\sigma)$. 
Let $\mw^\sigma$ be the corresponding symmetry-reduced flow, which solves \eqref{3DNStwistedcart}.
 Then the following hold:
\begin{enumerate}
\item There exists $C>0$, independent of $\sigma$, such that
\[
\begin{array}{l}
 \|\partial_t\mathbf{w}^{\sigma}\|_{L^2((0,T);H^{-1}(D))} 
 \leq C(\|\mathbf{w}^{\sigma}\|_{L^{\infty}((0,T);L^2(D))} + 1) \|\nabla\mathbf{w}^{\sigma}\|_{L^2((0,T);L^2(D))} . 
 \end{array}
\] \label{EnergyTwistedEqs.1}

\item There exists $C>0$, independent of $\sigma$, such that
\[
\begin{array}{l}
 \|q^{\sigma}\|_{L^2((0,T);L^2(D))} \leq C \left(\|\partial_t\mathbf{w}^{\sigma}\|_{L^2((0,T);H^{-1}(D))} \right. \\ 
 \qquad \qquad + \left. (\|\mathbf{w}^{\sigma}\|_{L^{\infty}((0,T);L^2(D))} + 1) \|\nabla\mathbf{w}^{\sigma}\|_{L^2((0,T);L^2(D))} \right) .
 \end{array}
\] \label{EnergyTwistedEqs.2}

\item Moreover, the following scaling holds:
\[
\|\mathbf{u}_0^{\sigma}\|_{L^2(\Omega^{\sigma})} = \sqrt{\sigma} \| \mathbf{w}_0^{\sigma}\|_{L^2(D)},
\] \label{EnergyTwistedEqs.3}

\item and we also have
\[
 \begin{gathered}
  \|\nabla_H\mathbf{u}_0^{\sigma}\|_{L^2(\Omega^{\sigma})} \leq \sqrt{\sigma} \| \nabla\mathbf{w}_0^{\sigma}\|_{L^2(D)}, \\
 \|\partial_{x_3}\mathbf{u}_0^{\sigma}\|_{L^2(\Omega^{\sigma})} \leq \frac{1}{\sqrt{\sigma}} \| \mathbf{w}_0^{\sigma}\|_{H^1(D)} .
\end{gathered}
\] \label{EnergyTwistedEqs.4}

\end{enumerate}
\end{proposition}

\begin{rem}  \label{EnergyTwistedEqs2}
As a result of Propositions \ref{EnergyTwistedEqsPrime} and \ref{EnergyTwistedEqs} it follows that
\begin{equation} \label{SummaryEnEstTwistedEqs}
\begin{array}{l}
\|\mathbf{w}^{\sigma}(t)\|_{L^2(D)} \leq \|\mathbf{w}_0^{\sigma}\|_{L^2(D)}, \;\;\; \mbox{ for each } t\in [0,T],\\ \\
\|\nabla\mathbf{w}^{\sigma}\|_{L^2((0,T);L^2(D))} \leq C\|\mathbf{w}_0^{\sigma}\|_{L^2(D)}, \\ \\
\|\partial_t\mathbf{w}^{\sigma}\|_{L^2((0,T);H^{-1}(D))} \leq C_1\|\mathbf{w}_0^{\sigma}\|_{L^2(D)}^2 + C_2\|\mathbf{w}_0^{\sigma}\|_{L^2(D)}, \\ \\
\|q^{\sigma}\|_{L^2((0,T);L^2(D))} \leq C_1\|\mathbf{w}_0^{\sigma}\|_{L^2(D)}^2 + C_2\|\mathbf{w}_0^{\sigma}\|_{L^2(D)},
\end{array}
\end{equation}
with constants that are uniform in $\sigma$ on $[1,+\infty)$.
\end{rem}

\begin{proof}
We begin with  estimate \eqref{EnergyTwistedEqs.1}. We recall that $\mathbf{u}^{\sigma}(x) = M^{\sigma}(x_3)\mathbf{w}^{\sigma}(t,m^{\sigma}(x_3) x')$, where $x' = (x_1, x_2)$. 
We exploit the duality between $H^{-1}$ and $H^1_0$ to compute
\[
    \|\pa_t \mw^\sigma\|_{H^{-1}(D)} 
   = \sup_{\boldsymbol{\Psi}\ne \mathbf{0}, \boldsymbol{\Psi}  \in H^1_0(D)} \frac{\langle \boldsymbol{\Psi}, \mw^\sigma\rangle}{\|\boldsymbol{\Psi}\|_{H^1_0}}.
\]
To this end, we test the  symmetry-reduced equations
\eqref{3DNStwistedcart} against a (vector) test function
$\boldsymbol{\Psi} \in H^1_0(D)^3$ and relate the weak form of the
reduced equations to that of the Navier-Stokes equations by
constructing an appropriate test function $\boldsymbol{\Phi}$ in
$H_{0,per}^1(\Omega^{\sigma})^3$ from$\boldsymbol{\Psi}$, as follows:
\[
\boldsymbol{\Phi}(x) \equiv \ds{\frac{1}{\sigma}}M^{\sigma}(x_3)\boldsymbol{\Psi}(m^{\sigma}(x_3) x').
\]
We recall now as well that $(m^\sigma(x_3))^{-1}y=
x'=(x_1,x_2)$  by \eqref{Yofx}.
We then observe that
\[
(M^{\sigma}(x_3))^{T}\boldsymbol{\Phi}((m^{\sigma})^{-1}(x_3) y ,x_3) = \frac{1}{\sigma}\boldsymbol{\Psi}(y) ,
\]
by the orthogonality of $M^\sigma$.

We have that
\[
\int_{D} \boldsymbol{\Psi} (y) \cdot \partial_t\mathbf{w}^{\sigma}(t,y)\,dy 
\]
\[
= \int_D \int_0^{\sigma}  (M^{\sigma}(x_3))^{T}\boldsymbol
{\Phi}((m^{\sigma})^{-1}(x_3) y,x_3) \cdot
\partial_t\mathbf{w}^{\sigma}(t,y)\,dx_3\, dy 
\]
\[
= \int_{\Omega^{\sigma}} \boldsymbol{\Phi}(x_H,x_3) \cdot  M^{\sigma}(x_3)\partial_t\mathbf{w}^{\sigma}(t,m^{\sigma}(x_3) x')\, dx =
\int_{\Omega^{\sigma}} \boldsymbol{\Phi}(x)
 \cdot \partial_t\mathbf{u}^{\sigma}(t,x)\,dx.
\]

To bound the $H^1$ norm of $\mbfu$, we  calculate the derivatives of $\boldsymbol{\Phi}$ to find
\[
\nabla_H\boldsymbol{\Phi}(x) = \frac{1}{\sigma} M^{\sigma}(x_3)\, [(D\Psi)(m^{\sigma}(x_3) x')] \,m^{\sigma}(x_3),
\]
\[
 \begin{aligned}
\partial_{x_3}\boldsymbol{\Phi} (x) &= \frac{2\pi}{\sigma^2}\left( 
\partial_{\rho} M^{\sigma}(x_3) \,\boldsymbol{\Psi}(m^{\sigma}(x_3) x') + 
\right . \\
 &\qquad \left. M^{\sigma}(x_3) \, [(D\boldsymbol{\Psi})(m^{\sigma}
(x_3)\,x_H)]\, [(\partial_{\rho}m^{\sigma})(x_3) x'] \right),
\end{aligned}
\]
where $D$ denotes differentiation of a function with respect to its
variables and $\rho$ denotes the argument of $M^\sigma$ and $m^\sigma$.
A simple change of variables then gives:
\[
\|\boldsymbol{\Phi}\|_{L^2(\Omega^{\sigma})} = \frac{1}{\sqrt{\sigma}}\|\boldsymbol{\Psi}\|_{L^2(D)}.
\]
\[
\|\nabla_H\boldsymbol{\Phi}\|_{L^2(\Omega^{\sigma})} \leq C\frac{1}{\sqrt{\sigma}}\|\nabla \boldsymbol{\Psi}\|_{L^2(D)},
\]
\[
\|\partial_{x_3}\boldsymbol{\Phi}\|_{L^2(\Omega^{\sigma})} \leq C\frac{1}{\sigma^{3/2}}\|\nabla\boldsymbol{\Psi}\|_{L^2(D)},
\]
with $C$  a constant independent of $\sigma$.

Hence, since $\sigma \geq 1$,
\begin{equation} \label{wtintermsut}
\|\partial_t\mathbf{w}^{\sigma}(t,\cdot)\|_{H^{-1}(D)} \leq C\frac{1}{\sqrt{\sigma}}\|\partial_t\mathbf{u}^{\sigma}(t,\cdot)\|_{H^{-1}(\Omega^{\sigma})}.
\end{equation}

Next, we estimate the $H^{-1}$ norm of $\pa_t \mbfu$ directly from  equations \eqref{3DNSper}:
\[
   \partial_t \mathbf{u}^{\sigma} =
   -\mathbb{P}[(\mathbf{u}^{\sigma}\cdot\nabla)\mathbf{u}^{\sigma}]+
   \mathbb{P}[\Delta \mathbf{u}^{\sigma}],
\]
where $\mathbb{P}$ denotes the Leray projector onto divergence-free vector fields tangent to $\partial D$ and periodic in $x_3$ with period
 $\sigma$, so that
\[
   \|\partial_t
   \mathbf{u}^{\sigma}(t,\cdot)\|_{H^{-1}(\Omega^{\sigma})}  \leq
   C_1\|\mbox{div }
 (\mathbf{u}^{\sigma}\otimes \mathbf{u}^{\sigma})(t,\cdot)\|_{H^{-1}
 (\Omega^{\sigma})} + C_2\|\Delta \mathbf{u}^{\sigma}(t,\cdot)\|_{H^{-1} (\Omega^{\sigma})}\]
 \[\leq C_1\|\mathbf{u}^{\sigma}(t,\cdot)\|_{L^4(\Omega^{\sigma})}^2 +
 C_2\|\nabla
 \mathbf{u}^{\sigma}(t,\cdot)\|_{L^2(\Omega^{\sigma})}\]
 \[= C_1\sqrt{\sigma}\|\mathbf{w}^{\sigma}(t,\cdot)\|_{L^4(D)}^2 +
 C_2\frac{1}{\sqrt{\sigma}}
 \|\nabla\mathbf{w}^{\sigma}(t,\cdot)\|_{L^2(D)},
\]
using the helical symmetry expressed by relation \eqref{wdefweak}. It follows from \eqref{wtintermsut} and the estimates above that
\[
\|\partial_t\mathbf{w}^{\sigma}(t,\cdot)\|_{H^{-1}(D)} \leq C\frac{1}{\sqrt{\sigma}}\left(\sqrt{\sigma}\|\mathbf{w}^{\sigma}(t,\cdot)\|_{L^4(D)}^2
+\frac{1}{\sqrt{\sigma}}\|\nabla\mathbf{w}^{\sigma}(t,\cdot)\|_{L^2(D)}\right)\]
\[\leq C\left( \|\mathbf{w}^{\sigma}(t,\cdot)\|_{L^2(D)} \|\nabla\mathbf{w}^{\sigma}(t,\cdot)\|_{L^2(D)}
+ \|\nabla\mathbf{w}^{\sigma}(t,\cdot)\|_{L^2(D)}\right),
\]
where we have used  the two-dimensional Ladyzhenskaya inequality \eqref{ladyineq2D}. This concludes the proof of estimate \eqref{EnergyTwistedEqs.1}.

To prove estimate 
\eqref{EnergyTwistedEqs.2},
we deal directly with the equations for $\mathbf{w}^{\sigma}$, $q^{\sigma}$. Since $p^\sigma$, and hence $q^{\sigma}$, is chosen up to a constant, we  can assume that
\[
\int_D q^{\sigma}(y) \, dy = 0.
\]
We again use duality and interpret 
the $L^2$-norm of $q^\sigma$  as the dual norm in
$(L^2(D))^\ast$. 
Consequently, we pick an arbitrary $f \in L^2(D)$ such that $\int_D f(y)\,dy = 0$
and \ $\|f\|_{L^2}=1$. 
By virtue of Lemma \ref{GaldisLemma} there exists $\mathbf{v}_H \in H^1_0(D)$ such that
\begin{equation} \label{vHtof}
\begin{array}{l}
\mbox{div }\mathbf{v}_H = f,\\ \\
\|\mathbf{v}_H\|_{H^1(D)}\leq C \|f\|_{L^2(D)} =  C.
\end{array}
\end{equation}

We multiply \eqref{3DNStwistedcart} by $\mathbf{v}_H$ and integrate over $D$ to find:
\begin{equation} \label{EnergyIdentTwistedvH}
\begin{gathered}
\ds{\int_D \mathbf{v}_H \cdot \left(
\partial_t \mathbf{w}^{\sigma}_H + (\mathbf{w}_H^{\sigma}\cdot\nabla_y)\mathbf{w}^{\sigma}_H + \ds{\frac{2\pi}{\sigma}}w^{\sigma,3}[E\mathbf{w}^{\sigma}_H - (\mathbf{w}_H^{\sigma})^{\perp}] \right) \,dy} \\ \\
\;\;\;\;\;\;\; \ds{= \int_D \mathbf{v}_H }\cdot \left(- (\nabla_y q^{\sigma})_H + \Delta_y \mathbf{w}^{\sigma}_H + \ds{\frac{4\pi^2}{\sigma^2}} [E^2\mathbf{w}^{\sigma}_H \right.\\ \\
\qquad \qquad\left.- 2 E (\mathbf{w}_H^{\sigma})^{\perp} - \mathbf{w}_H^{\sigma}]\right)\,dy , \\ \\
\end{gathered}
\end{equation}

We next perform several integrations by parts, using the divergence constraint for $\mw^\sigma$:
\[
\mathrm{div}_y \mathbf{w}_H + \ds{\frac{2\pi}{\sigma}} E w^3 = 0,
\]
together with \eqref{vHtof}, to find 
\begin{equation} \label{EnergyIdentTwisted}
\begin{gathered}
\ds{\int_D \mathbf{v}_H \cdot \partial_t \mathbf{w}\,dy  -  \int_D \mathbf{w}^{\sigma}\cdot[(\mathbf{w}_H^{\sigma}\cdot\nabla_y)\mathbf{v}_H ]\,dy}
 \\ \\
\qquad \qquad - \ds{\frac{2\pi}{\sigma}}\int_D \mathbf{w}_H^{\sigma}\cdot w^{3,\sigma} E\mathbf{v}_H + w^{3,\sigma}\mathbf{v}_H \cdot (\mathbf{w}_H^{\sigma})^{\perp}  \,dy \\ \\
\;\;\;\;\;\;\; \ds{ = \int_D f(y)\,q^{\sigma}(y)\,dy -\int_D \nabla_y \mathbf{v}_H\cdot\nabla_y \mathbf{w}^{\sigma}\,dy } -
\ds{\frac{4\pi^2}{\sigma^2}} \ds{\int_D E\mathbf{v}_H\cdot E\mathbf{w}^{\sigma}_H \,dy}\\ \\
\qquad \qquad + \ds{\frac{8\pi^2}{\sigma^2}} \int_D E\mathbf{v}_H \cdot (\mathbf{w}_H^{\sigma})^{\perp}\,dy - \ds{\frac{4\pi^2}{\sigma^2}}\int_D \mathbf{v}_H\cdot \mathbf{w}_H^{\sigma}\,dy . \\ \\
\end{gathered}
\end{equation}
By Poincar\'e's inequality for functions with zero average on $D$, we deduce that
\begin{equation} \label{qEst1}
 \begin{aligned}
  \left| \int_D f \,q^{\sigma}\,dy \right| \leq 
  C\|\mathbf{v}_H\|_{H^1(D)}&(\| 
  \partial_t\mathbf{w}^{\sigma}\|_{H^{-1}(D)} + \\ 
  & \quad \qquad \|\mathbf{w}^{\sigma}\|_{L^4(D)}^2 +
  \|\nabla\mathbf{w}^{\sigma}\|_{L^2(D)}),
 \end{aligned}
\end{equation}
for $C$ a constant independent of $f$ or $\sigma$.
Above we exploit that 
the operator $E=y^\perp \cdot \nabla_y$ is first order and $\sigma \geq 1$.

Hence, using that $\|\mathbf{v}_H\|_{H^1(D)} \leq C\|f\|_{L^2(D)}= C $
from \eqref{vHtof} and  the Ladyzhenskaya inequality again, we find
\begin{equation} \label{qEst2}
\| q^{\sigma}\|_{L^2(D)} \equiv \left| \int_D f \,q^{\sigma}\,dy \right|
 \leq C(\|\partial_t\mathbf{w}^{\sigma}\|_{H^{-1}} + \|\mathbf{w}^{\sigma}\|_{L^2}\|\nabla\mathbf{w}^{\sigma}\|_{L^2} +
\|\nabla\mathbf{w}^{\sigma}\|_{L^2})
\end{equation}

Finally, squaring both sides of the inequality \eqref{qEst2} and using Young's inequality, subsequently integrating in time, we arrive at
\begin{equation} \label{qEstFINAL}
\begin{aligned}
  \| q^{\sigma}&\|_{L^2((0,T);L^2(D))}^2 \leq C(\|\partial_t\mathbf{w}^{\sigma}
\|_{L^2((0,T);H^{-1}(D))}^2 \\ 
& \quad + (\|\mathbf{w}^{\sigma}\|_{L^{\infty}((0,T);L^2(D))}^2 + 1)
\|\nabla\mathbf{w}^{\sigma}
\|_{L^2((0,T);L^2(D))}^2).
\end{aligned}
\end{equation}

Identities \eqref{EnergyTwistedEqs.3} and 
\eqref{EnergyTwistedEqs.4}
 follow by a straightforward change of variables, 
from the relation
\[
\mathbf{u}_0^{\sigma}(x) = M^{\sigma}(x_3)\mathbf{w}_0^{\sigma}(m^{\sigma}(x_3) x'),
\]
which gives by the chain rule, 
\[
\nabla_H \mathbf{u}_0^{\sigma} = M^{\sigma}(x_3)
[(D\mathbf{w}_0^{\sigma})(m^{\sigma}(x_3) x')]
[m^{\sigma}(x_3)],
\]
and
\begin{multline*}
\partial_{x_3}\mathbf{u}_0^{\sigma} =
\frac{2\pi}{\sigma}\left[ \partial_{\rho} M^{\sigma}(x_3) 
\mathbf{w}_0^{\sigma}(m^{\sigma}(x_3) x') \right.\\
\left. +
  M^{\sigma}(x_3)[(D\mathbf{w}_0^{\sigma})(m^{\sigma}(x_3))][\partial_{\rho}m^{\sigma}(x_3)] 
x'\right].
\end{multline*}
\end{proof}

With these estimates at hand, we are now ready to discuss the limit $\sigma \to \infty$. We observe that $\sigma$ is not a parameter appearing explicitly in the Navier-Stokes system 
\eqref{3DNSper}. Therefore it is not clear what the limit equations are even at a formal level. 
The dependence on $\sigma$ is elucidated however in the
symmetry-reduced system \eqref{3DNStwistedcart}, which is equivalent
to the original system at the level of strong solutions 
thanks to Proposition \ref{helicalYvariables}.

For the reduced system \eqref{3DNStwistedcart},  formally setting $\sigma =\infty$ produces the following system of equations for a three-component vector function
 $\mw^\infty: (0,+\infty) \times D\to \RR^3$, with associated pressure $q^\infty$:
\begin{equation} \label{2D3compNS}
\begin{cases}
\partial_t w^{\infty,1}  + (w^{\infty,1}\partial_{y_1}+w^{\infty,2}
\partial_{y_2}) w^{\infty,1} = -\partial_{y_1} q^\infty + (\partial_{y_1}^2 +\partial_{y_2}^2) w^{\infty,1 }, 
%& \mbox{ in }(0,+\infty) \times D ; 
\\ 
\partial_t w^{\infty,2}  + (w^{\infty,1} \partial_{y_1}+w^{\infty,2}
\partial_{y_2}) w^{\infty,2}= -\partial_{y_2} q^\infty + (\partial_{y_1}^2 +\partial_{y_2}^2) w^{\infty,2} , 
%& \mbox{ in } (0,+\infty) \times D ; 
\\ 
\partial_t w^{\infty,3}  + \left( w^{\infty,1} \partial_{y_1}+w^{\infty,2}
\partial_{y_2}\right) \, w^{\infty,3} =  (\partial_{y_1}^2 +\partial_{y_2}^2) w^{\infty,3}, 
%& \mbox{ in } (0,+\infty) \times D ; 
\\
\partial_{y_1} w^{\infty,1}  + \partial_{y_2} w^{\infty,2}  = 0, \qquad  \qquad  \qquad\qquad \quad \;\mbox{ in } \; [0,+\infty) \times D ; \\
\mathbf{w}^\infty = 0, \qquad \qquad \qquad \qquad \qquad \qquad \qquad \qquad  \mbox{ on } \; [0,+\infty) \times \partial D ;\\ 
\mathbf{w}^\infty(0,y) = \mathbf{w^\infty}_0(y), \qquad \qquad \qquad \qquad \qquad \quad y\in D.
\end{cases}
\end{equation}
The initial condition $\mw^\infty_0$ will be taken in $H^1_0(D)$ and assumed to satisfy:
\begin{equation} \label{straightdivfree} 
\partial_{y_1} w^{\infty,1}_0 + \partial_{y_2} w^{\infty,2}_0 = 0.
\end{equation}

The first two momentum equations are independent
of $w^{\infty,3}$ and together with the fourth equation give precisely
the two-dimensional Navier-Stokes equations in $D$, where the fluid
velocity is identified with $\mw^\infty_H:=
(w^{\infty,1},w^{\infty,2},0)$. The third component $w^{\infty,3}$ is
simply advected by the first two and diffused. 
For this reason, we refer to this flow as a planar flow.
Existence and regularity results for the 2D Navier-Stokes equations  immediately give existence and uniqueness of the divergence-free vector field $\mw^\infty_H\in C([0,T);H^1_0(D))\cap L^2((0,T);H^2(D)\cap H^1_0(D))$ and associated pressure $q^\infty\in
L^2((0,T);H^1(D))\cap C^\infty((0,T)\times D)$ for any  initial condition $\mw^\infty_H(0)\in H^1_0(D)$ satisfying  \eqref{straightdivfree}, and any $T>0$. In fact, $\mw_H^\infty$ is smooth for $t>0$.
Consequently, the advection-diffusion equation for $w^{\infty,3}$
admits a unique solution, which belongs to the same class (see
e.g. Proposition 2.7 in \cite{MajdaBertozzi2002} and 
Theorem 3.10 in \cite{T}.) We refer to the three-component vector function 
\[
  \mw^\infty \in C([0,T);H^1_0(D))\cap  
  L^2((0,T);H^2(D)\cap H^1_0(D))\cap 
  C^\infty((0,T)\times D),
\]
as the unique strong solution of  problem \eqref{2D3compNS}.

The  System \eqref{2D3compNS} gives the so-called two-dimensional, three-component
Navier-Stokes equations (also known as the {\em $2\frac{1}{2}$D
  Navier-Stokes equations}, see \cite{MajdaBertozzi2002}.) 
We can uniquely associate to $\mw^\infty$ a solution
$\mbfu^\infty$ of the Navier-Stokes equations in $\Omega$ with initial
data $\mbfu_0^\infty$ by:
\begin{equation} \label{uinftydef}
 \begin{aligned}
    &\mbfu^\infty(t,x)  := \mw^\infty(t, x'), & x'
   \in D, \; t>0, \\
   &\mbfu^\infty_0(x)  := \mw^\infty_0(x'), & x'
   \in D \\
  &p^\infty(t,x)  := q^\infty(t,x',0), &  x'
   \in D, \; t>0,
\end{aligned}
\end{equation}
with $x'=(x_1,x_2)$. It is immediate to see that $\mbfu^\infty$ and $p^\infty$ have at least the same regularity as $\mw^\infty$ and $q^\infty$.
We will refer to $\mbfu^\infty$ as the $2\frac{1}{2}$D solution of the
Navier-Stokes system   \eqref{3DNSper} in $\Omega$
with associated pressure $p^\infty$. 

To obtain a relationship with the original problem 
\eqref{3DNSper}, at least at a formal level, 
we observe that, if we  take the limit $\sigma\to \infty$ in \eqref{uandpinY}, thanks to \eqref{Mofx3} and \eqref{Yofx},  we have the identification:
\begin{equation} \label{eq.sigmalim}
    \mbfu^\infty(t,x) = \mw^\infty(t,x') \equiv \lim_{\sigma\to\infty}     
    \mbfu^\sigma(t,x)
\end{equation}
Above, we have naturally identified the cross section of the cylinder $\Omega$ at height $x_3=0$ with $D$ and $x'=(x_1,x_2)$ with $y$.
%In fact, we can think of $\mw^\infty_0$ as a vector field in $\Omega$, which is constant in the direction of the pipe, that is, independent of $x_3$, and, as such, $\mw^\infty_0$ is divergence free. This ansatz is then preserved by the Navier-Stokes evolution if we assume the pressure $q^\infty$ is also constant in $x_3$.   
We will use the identities and estimates established in Proposition \ref{EnergyTwistedEqs}, valid for all  $1 \leq \sigma < \infty$, to establish an estimate for the difference between $\mathbf{w}^{\sigma}$ and $\mathbf{w}^{\infty}$. 
One difficulty in studying the limit $\sigma \to \infty$ is that $\mw^\infty_H$ is divergence free, while $\mw^\sigma_H$ satisfies a divergence constraint that is $\sigma$ dependent.

\begin{proposition} \label{EnergyEstDiff}
% Let $\mathbf{u}_0 \in H^1_0(D)$ satisfy \eqref{straightdivfree} and, for each $\sigma \geq 1$, choose $\mathbf{w}_0^{\sigma} \in H^1_0(D)$ satisfying \eqref{TwistedDivFreeCond}. 
Let $\mw^\infty_0\in H^1_0(D)$ satisfy \eqref{straightdivfree}. Given $\sigma \geq 1$, let 
$\mathbf{w}_0^{\sigma} \in H^1_0(D)$ satisfy \eqref{TwistedDivFreeCond}. 
Let $\mathbf{w}^\infty$ be the unique strong solution of \eqref{2D3compNS} with initial data 
$\mathbf{w}^\infty_0$, and let $\mathbf{w}^{\sigma}$ be the unique
regular solution of  \eqref{3DNStwistedcart} with initial data
$\mathbf{w}_0^{\sigma}$ on the  time interval $(0,T)$, $T>0$.
Set:
\begin{equation} \label{eq.EnergyEstDiff}
  \boldsymbol{\boldsymbol{\Theta}}^{\sigma} \equiv \mathbf{w}^{\sigma} - \mw^\infty.
\end{equation}
Then,  for all $0<t<T$,
\begin{equation} \label{ThetaEnergyEst}
\begin{aligned}
\int_D |\boldsymbol{\boldsymbol{\Theta}}^{\sigma}(t,y)|^2 \, dy &+ \int_0^t\int_D |\nabla \boldsymbol
{\boldsymbol{\Theta}}^{\sigma}(s,y)|^2\,ds\,dy \\
 &\leq C\left(t,\|\mathbf{w}^\infty_0\|_{L^2}^2,\|\mathbf{w}_0^{\sigma}\|_{L^2}^2\right) \left( \|\boldsymbol{\boldsymbol{\Theta}}^{\sigma}_0\|_{L^2}^2 + \frac{1}{\sigma}\right).
\end{aligned}
\end{equation}
\end{proposition}

\begin{proof}
Since $\mw^\sigma$ is a strong solution of \eqref{3DNStwistedcart} and
$\mw^\infty$ is of 
\eqref{2D3compNS}  on the interval $[0,T)$, there exist functions $q^\sigma$ and 
$q^\infty \in L^1((0,T);H^1(D))$
enforcing the divergence  constraints. If we set 
\ $r^\sigma=q^\infty -q^\sigma$, then $\boldsymbol{\Theta}^\sigma$
satisfies 
 the following set of equations on $(0,T)\times D$:
\begin{equation} \label{ThetaEqs}
\left\{
\begin{array}{ll}
\partial_t \boldsymbol{\Theta}^{\sigma} + (\mathbf{w}^{\sigma}_H\cdot\nabla_y)\boldsymbol{\Theta}^{\sigma} + (\boldsymbol{\Theta}^{\sigma}_H\cdot\nabla_y)\mathbf{w}^\infty +
\ds{\frac{2\pi}{\sigma}}w^{\sigma,3}[E\mathbf{w}^{\sigma} - (\mathbf{w}^{\sigma}_H)^{\perp}] \\ \\
\;\;\;\;\;\;\;= -(\nabla_y r^{\sigma})_H -\ds{\frac{2\pi}{\sigma}}Eq^{\sigma}\mathbf{e_3}
+ \Delta_y \boldsymbol{\Theta}^{\sigma} + \ds{\frac{4\pi^2}{\sigma^2}} [E^2\mathbf{w}^{\sigma}
- 2 E (\mathbf{w}^{\sigma}_H)^{\perp} - \mathbf{w}^{\sigma}_H], \\ \\
\mathrm{div}_y \boldsymbol{\Theta}^{\sigma}_H + \ds{\frac{2\pi}{\sigma}} E w^{\sigma,3} = 0,
\end{array}
\right.
\end{equation}
where $E$ is again the differential operator 
\ $y^\perp \cdot \nabla_y$ defined in \eqref{E1}.
These equations are complemented by the initial condition 
\[
     \boldsymbol{\Theta}^\sigma_0:=
    \boldsymbol{\Theta}^\sigma(0) = 
    \mw^\sigma_0 - \mw^\infty_0 \in H^1_0(D)
\]
and no-slip boundary conditions on $\pa D$.

We observe that $\boldsymbol{\Theta}^\sigma$ has enough regularity to be a test function for the weak formulation of \eqref{ThetaEqs}. In particular, $\pa_t \boldsymbol{\Theta}^\sigma\in L^2((0,T),L^2(D))$.
The weak form, after rearranging the terms and integrating by parts, gives:
\begin{align} \label{ThetaDifflIdent}
%\begin{array}{l}
&\ds{\frac{1}{2}\frac{d}{dt}\int_D |\boldsymbol{\Theta}^{\sigma}|^2\,dy + \int_D |\nabla\boldsymbol{\Theta}^{\sigma}|^2\,dy
= -\int_D \boldsymbol{\Theta}^{\sigma}\cdot \left[ (\mathbf{w}^{\sigma}_H\cdot\nabla_y)\boldsymbol{\Theta}^{\sigma}\right]\,dy} \nonumber \\[2mm] 
&\ds{- \int_D \boldsymbol{\Theta}^{\sigma}\cdot \left[ (\boldsymbol{\Theta}^{\sigma}_H\cdot\nabla_y)\mathbf{w}^\infty \right]\,dy
-\frac{2\pi}{\sigma} \int_D \boldsymbol{\Theta}^{\sigma}\cdot \left[ w^{\sigma,3}(E\mathbf{w}^{\sigma} - (\mathbf{w}^{\sigma}_H)^{\perp}) \right]\,dy}\nonumber\\[2mm] 
&\ds{+\frac{4\pi^2}{\sigma^2} \int_D \boldsymbol{\Theta}^{\sigma}\cdot \left[ E^2\mathbf{w}^{\sigma}
- 2 E (\mathbf{w}^{\sigma}_H)^{\perp} - \mathbf{w}^{\sigma}_H \right]\,dy} \\[2mm] 
&\ds{-\int_D \boldsymbol{\Theta}^{\sigma}\cdot \left[ (\nabla_y r^{\sigma})_H + \frac{2\pi}{\sigma}Eq^{\sigma}\mathbf{e_3}  \right]\,dy \equiv -I_1-I_2
-I_3+I_4-I_5.}\nonumber
%\end{array}
\end{align}

We estimate each of the five integrals on the right-hand side. 
Since $\mw^\sigma$ is a strong  solution and in view of estimates 
\eqref{SummaryEnEstTwistedEqs} for $\mw^\sigma$, all norms appearing
in the bounds to 
follow are finite and all constants $C$ are uniform in $\sigma \in [1,+\infty)$.
We have:
\[
 \begin{aligned}
   & |I_1| \leq \frac{2\pi}{\sigma}\int_D \frac{1}{2}
|\boldsymbol{\Theta}^{\sigma}|^2|\nabla \mathbf{w}^{\sigma}|\,dy, \\
   & |I_2| \leq \int_D |\boldsymbol{\Theta}^{\sigma}|^2|\nabla
   \mathbf{w}^\infty|\,dy, \\
   &|I_3| \leq \frac{2\pi}{\sigma}\int_D
   |\boldsymbol{\Theta}^{\sigma}||\mathbf{w}^{\sigma}|(|\nabla\mathbf{w}^{\sigma}|
   + |\mathbf{w}^{\sigma}|)\,dy, \\
    &|I_4| \leq \frac{4\pi^2}{\sigma^2} \left[\int_D
      |\nabla\boldsymbol{\Theta}^{\sigma}|
      |\nabla\mathbf{w}^{\sigma}|\,dy + \int_D 
      |\boldsymbol{\Theta}^{\sigma}|(|\nabla\mathbf{w}^{\sigma}| + 
   |\mathbf{w}^{\sigma}|)\,dy \right], \\
   &|I_5| = \left| \frac{2\pi}{\sigma} \int_D [q^\infty\, \,E(w^{\sigma,3}) - q^{\sigma}\,E(w^{\infty,3})]\,dy \right|\leq \frac{2\pi}{\sigma}\int_D
( |\nabla\mathbf{w}^{\sigma}|\,|q^\infty| + |\nabla \mathbf{w^\infty}|
\,|q^{\sigma}|)\,dy.
\end{aligned}
\]

We bound further each integral $I_i$, $i=1,\ldots,5$, using repeatedly the Ladyzhenskaya inequality
\eqref{ladyineq2D}, Cauchy-Schwartz and Young's inequalities:
\begin{align} 
 &|I_1| \leq
\frac{C}{\sigma} \|\boldsymbol{\Theta}^{\sigma}\|_{L^4(D)}^2\|\nabla \mathbf{w}^{\sigma}\|_{L^2(D)}
\leq
\frac{C}{\sigma}
(\|\boldsymbol{\Theta}^{\sigma}\|_{L^2(D)}^2\|\nabla
\mathbf{w}^{\sigma}\|_{L^2(D)}^2 +
\|\nabla\boldsymbol{\Theta}^{\sigma}\|_{L^2(D)}^2), \label{I1}\\ 
 &|I_2| \leq
\|\boldsymbol{\Theta}^{\sigma}\|_{L^4(D)}^2\|\nabla \mathbf{w}^\infty\|_{L^2(D)}
\leq
\frac{1}{2} \|\boldsymbol{\Theta}^{\sigma}\|_{L^2(D)}^2\|\nabla
\mathbf{w}^\infty\|_{L^2(D)}^2 + \frac{1}{2}
\|\nabla\boldsymbol{\Theta}^{\sigma}\|_{L^2(D)}^2, \label{I2} 
\end{align}
\begin{align}
&|I_3| \leq
\ds{\frac{C}{\sigma}(\|\boldsymbol{\Theta}^{\sigma}\|_{L^2(D)}\|\mathbf{w}^{\sigma}\|_{L^4(D)}^2 +  \|\boldsymbol{\Theta}^{\sigma}\|_{L^4(D)} \|\nabla\mathbf{w}^{\sigma}\|_{L^2(D)}
\|\mathbf{w}^{\sigma}\|_{L^4(D)})}  \nonumber \\
&\leq
\ds{\frac{C}{\sigma}(\|\boldsymbol{\Theta}^{\sigma}\|_{L^2(D)}^2 \|\mathbf{w}^{\sigma}\|_{L^2(D)}^2
+ \|\nabla \mathbf{w}^{\sigma}\|_{L^2(D)}^2 + \|\boldsymbol{\Theta}^{\sigma}\|_{L^2(D)}^{1/2} \|\nabla \boldsymbol{\Theta}^{\sigma}\|_{L^2(D)}^{1/2} 
\|\nabla\mathbf{w}^{\sigma}\|_{L^2(D)}^{3/2}\|\mathbf{w}^{\sigma}\|_{L^2(D)}^{1/2})}
\nonumber \\ 
&\leq
\ds{\frac{C}{\sigma}(\|\boldsymbol{\Theta}^{\sigma}\|_{L^2(D)}^2 \|\mathbf{w}^{\sigma}\|_{L^2(D)}^2
+ \|\nabla \mathbf{w}^{\sigma}\|_{L^2(D)}^2 + \|\boldsymbol{\Theta}^{\sigma}\|_{L^2(D)}^2 \|\nabla \boldsymbol{\Theta}^{\sigma}\|_{L^2(D)}^2 \|\mathbf{w}^{\sigma}\|_{L^2(D)}^2
 + \|\nabla\mathbf{w}^{\sigma}\|_{L^2(D)}^2)} \nonumber\\ 
&\leq
\ds{\frac{C}{\sigma}}[\|\boldsymbol{\Theta}^{\sigma}\|_{L^2(D)}^2
\|\mathbf{w}^{\sigma}\|_{L^2(D)}^2 + \|\nabla
\mathbf{w}^{\sigma}\|_{L^2(D)}^2 \label{I3} \\ 
&\quad \qquad + (\|\mathbf{w}^{\sigma}\|_{L^2(D)}^2 + \|\mathbf{u}\|_{L^2(D)}^2) \|\nabla \boldsymbol{\Theta}^{\sigma}\|_{L^2(D)}^2 \|\mathbf{w}^{\sigma}\|_{L^2(D)}^2
+ \|\nabla\mathbf{w}^{\sigma}\|_{L^2(D)}^2],  \nonumber
\end{align} 
\begin{align}
&|I_4| \leq  \frac{C}{\sigma^2} 
( \|\nabla\boldsymbol{\Theta}^{\sigma}\|_{L^2}\|\nabla\mathbf{w}^{\sigma}\|_{L^2(D)}  +   \|\boldsymbol{\Theta}^{\sigma}\|_{L^2(D)}\|\nabla\mathbf{w}^{\sigma}\|_{L^2(D)} 
+
\|\boldsymbol{\Theta}^{\sigma}\|_{L^2(D)}\|\mathbf{w}^{\sigma}\|_{L^2(D)}) \nonumber
 \\
&\qquad \leq
\frac{C}{\sigma^2} ( \|\nabla\boldsymbol{\Theta}^{\sigma}\|_{L^2(D)}^2
+ \|\nabla\mathbf{w}^{\sigma}\|_{L^2(D)}^2  +   \|\boldsymbol{\Theta}^{\sigma}\|_{L^2(D)}^2 
+  \|\mathbf{w}^{\sigma}\|_{L^2(D)}^2), \label{I4}
\end{align}
\begin{align} 
&|I_5|
\leq \frac{C}{\sigma}
  (\|q^\infty\|_{L^2(D)}\|\nabla\mathbf{w}^{\sigma}\|_{L^2(D)} +
  \|q^{\sigma}\|_{L^2(D)}\|\nabla \mathbf{w}^\infty\|_{L^2}(D))
  \nonumber  \\
&\qquad \leq \frac{C}{\sigma} (\|q^\infty\|_{L^2(D)}^2 +
  \|\nabla\mathbf{w}^{\sigma}\|_{L^2(D)}^2 + \|q^{\sigma}\|_{L^2(D)}^2
  + \|\nabla \mathbf{w}^\infty\|_{L^2(D)}^2).   \label{I5} 
\end{align}

Inserting estimates \eqref{I1} --- \eqref{I5} into identity \eqref{ThetaDifflIdent} yields: 
\begin{align}
&\frac{d}{dt}\|\boldsymbol{\Theta}^{\sigma}\|_{L^2(D)}^2 + \|\nabla\boldsymbol{\Theta}^{\sigma}\|_{L^2(D)}^2
\leq \|\nabla\boldsymbol{\Theta}^{\sigma}
\|_{L^2(D)}^2 \cdot \nonumber \\
&\qquad \cdot \ds{\left(
\frac{C}{\sigma} + \frac{1}{2} + \frac{C}{\sigma}\|\mathbf{w}^{\sigma}\|_{L^2(D)}^4 + \frac{C}{\sigma}\|\mathbf{w}^{\sigma}\|_{L^2(D)}^2\|\mathbf{w}^\infty\|_{L^2(D)}^2
+ \frac{C}{\sigma^2}
\right) + \|\boldsymbol{\Theta}^{\sigma}\|_{L^2(D)}^2} \cdot
\nonumber \\
&\quad\ds{\cdot 
\left(
\frac{C}{\sigma} \|\nabla\mathbf{w}^{\sigma}\|_{L^2(D)}^2 + \frac{C}{\sigma}\|\nabla \mathbf{w}^\infty\|_{L^2(D)}^2 + \frac{C}{\sigma} \|\mathbf{w}^{\sigma}\|_{L^2(D)}^2 +
\frac{1}{2} \|\nabla \mw^\infty\|^2_{L^2(D)} +
\frac{C}{\sigma^2}
\right)}  \nonumber \\
&\qquad \ds{+ \left(
\frac{C}{\sigma} \|\nabla\mathbf{w}^{\sigma}\|_{L^2(D)}^2 + \frac{C}{\sigma^2} 
\|\nabla\mathbf{w}^{\sigma}\|_{L^2(D)}^2
+ \frac{C}{\sigma^2} \|\mathbf{w}^{\sigma}\|_{L^2(D)}^2 + \frac{C}{\sigma} \|\nabla \mathbf{w}^\infty\|_{L^2(D)}^2
\right.
}  \nonumber \\
& \qquad \qquad \qquad \qquad \qquad  \qquad\ds{
+ \left. \frac{C}{\sigma} \|q^\infty\|_{L^2(D)}^2 +
\frac{C}{\sigma}\|q^{\sigma}\|_{L^2(D)}^2
\right).}  \label{DiffIneqTheta}
\end{align}

Thanks again to the regularity of $\mw^\infty$, i.e.,
$$
\mw^\infty \in C([0,T);H^1_0(D))\cap L^2((0,T);H^2(D)\cap H^1_0(D)),
$$ and 
estimates \eqref{SummaryEnEstTwistedEqs} for $\mw^\sigma$, we can now choose $\sigma$ large enough such that
\[
\frac{C}{\sigma} + \frac{1}{2} + \frac{C}{\sigma}\|\mathbf{w}^{\sigma}\|_{L^2(D)}^4 + \frac{C}{\sigma}\|\mathbf{w}^{\sigma}\|_{L^2(D)}^2\|\mathbf{w}^\infty\|_{L^2(D)}^2
 + \frac{C}{\sigma}^2 < \frac{3}{4}.
\]
We will rewrite \eqref{DiffIneqTheta} as a differential inequality in
order to apply  Gr\"onwall's Lemma. 
To this end, we introduce the functions
\[
 f(t) = \frac{C}{\sigma} \|\nabla\mathbf{w}^{\sigma}\|_{L^2}^2 + 
\left(\frac{C}{\sigma}+\frac{1}{2}\right)\|\nabla \mathbf{w}^\sigma\|_{L^2}^2 + \frac{C}{\sigma} \|\mathbf{w}^{\sigma}\|_{L^2}^2 +
\frac{C}{\sigma^2} 
\] 
and
\[
g(t) = \frac{C}{\sigma} \|\nabla\mathbf{w}^{\sigma}(t)\|_{L^2}^2 + \frac{C}{\sigma^2} \|\nabla\mathbf{w}^{\sigma}
(t)\|_{L^2}^2
+ \frac{C}{\sigma^2} \|\mathbf{w}^{\sigma}(t)\|_{L^2}^2 \qquad \qquad \]
\[
\qquad \qquad \qquad + \frac{C}{\sigma} \|\nabla \mathbf{w}^\infty(t)\|_{L^2}^2
+ \frac{C}{\sigma} \|q^\infty(t)\|_{L^2}^2 + \frac{C}{\sigma}\|q^{\sigma}(t)\|_{L^2}^2.
\]
We also set
\[
z(t) = \|\boldsymbol{\Theta}^{\sigma}(t)\|_{L^2}^2.
\]
With this notation the differential inequality above becomes
\[
\frac{d}{dt}z \leq f(t)z + g(t),
\]
so that, by Gr\"onwall's Lemma we conclude that
\begin{equation} \label{zEstimate}
z(t) \leq \exp\left\{\int_0^t f(s)\,ds\right\} \, z(0) + \int_0^t \exp\left\{\int_s^t f(\tau)\,d\tau\right\} \, g(s) \, ds.
\end{equation}

Next, standard energy estimates for the 2D Navier-Stokes equations
along with similar energy estimates for advection-diffusion equations,
using that $\mw^\infty_H$ is divergence-free, 
give:
\[
   \int_0^t \|\nabla \mw^\infty(s)\|_{L^2(D)}^2\,ds
   \leq  \|\mw^\infty_0\|_{L^2(D)}^2.
\]
We employ once again the estimates 
\eqref{SummaryEnEstTwistedEqs}  
in Remark \ref{EnergyTwistedEqs2}
to  deduce that
\[\int_0^t f(s) \, ds \leq C\,
\left(\|\mathbf{w}^\infty_0\|_{L^2}^2 + 
\frac{1}{\sigma}
\|\mathbf{w}_0^{\sigma}\|_{L^2}^2\right),
\]
\[
\int_0^t g(s) \, ds \leq \frac{C}{\sigma}\left(\|\mathbf{w}^\infty_0\|_{L^2}^2 + 
\|\mathbf{w}_0^{\sigma}\|_{L^2}^2\right).
\]

Hence  we arrive at the estimate, 
using that $\sigma\geq 1$,
\[
\|\boldsymbol{\Theta}^{\sigma}\|_{L^2}^2 \leq C\left(\|\mathbf{w}^\infty_0\|_{L^2}^2,\|\mathbf{w}_0^{\sigma}\|_{L^2}^2\right) \|\boldsymbol{\Theta}^{\sigma}_0\|_{L^2}^2 + \frac{C}{\sigma}\left(\|\mathbf{w}^\infty_0\|_{L^2}^2 + \|\mathbf{w}_0^{\sigma}\|_{L^2}^2\right).
\]

This estimate, together with the choice of $\sigma$, produces, upon integrating the differential inequality \eqref{DiffIneqTheta} in time, the desired result.
\end{proof}

Before formulating our main results concerning the limit $\sigma \to \infty$ , we note a consequence of Proposition \ref{EnergyEstDiff}; namely, there may be more than one $2\frac{1}{2}$D flow within a certain distance to a given helical flow $\mw^\sigma$.
This non-uniqueness will be apparent later, since a correction to the initial data $\mw^\sigma_0$ will be needed to enforce the divergence-free condition for $\mw^\infty_0$.

We start with a simpler result,
describing a way in which solutions of the two-dimensional, three-component Navier-Stokes
equations can be approximated by suitable helical solutions of the three-dimensional
Navier-Stokes equations.
More precisely, suppose we are initially given a vector function
\[
 \mathbf{w}^\infty_0 = (w_0^{\infty,1},w_0^{\infty,2},w_0^{\infty,3}) \in H^1_0 (D)
\]
that satisfies the divergence-free constraint
\eqref{straightdivfree}. Let $\mw^\infty$ be the unique strong solution of \eqref{2D3compNS} with pressure $q^\infty$. 
Recall that we can uniquely associate to $\mw^\infty$ a solution
$\mbfu^\infty$ of the Navier-Stokes equations in $\Omega$ with initial
data $\mbfu_0^\infty$ via \eqref{uinftydef}.
We will construct a $\sigma$-dependent correction to
$\mathbf{w}^\infty_0$, $\mathbf{v}_0^{\sigma}$, using Lemma
\ref{GaldisLemma}, so that the resulting field $\mathbf{w}^\sigma_0$,
given in \eqref{w0sigma} below, 
 satisfies \eqref{TwistedDivFreeCond} and, hence, can be taken as initial data for the reduced helical equations \eqref{3DNStwistedcart}.

We first observe that, since $\mw^\infty_0\vert_{\pa D}\equiv 0$ and $\dive_y y^\perp\equiv 0$, 
\[
\int_D  y^\perp \cdot \nabla_y w_0^{\infty,3} \,dy = \int_D \mbox{ div}_y \, (y^{\perp}\,w_0^{\infty,3})\,dy = 0.
\]
Therefore, by Lemma \ref{GaldisLemma}, there exists a solution \  $\mathbf{v}_0^{\sigma}=(v_0^{\sigma,1},v_0^{\sigma,2}) \in H^{1}_0(D)$
to the problem
\begin{equation} \label{correctionu}
\mbox{div}_y \, \mathbf{v}_0^{\sigma} = -\frac{2\pi}{\sigma} E\,w_0^{\infty,3},
\end{equation}
such that
\begin{equation} \label{correctionuestimate}
\|\mathbf{v}_0^{\sigma}\|_{H^1} \leq C\frac{1}{\sigma}\|E \, w^{\infty,3}_0\|_{L^2(D)}\leq \frac{C}{\sigma}\|\mathbf{w}^\infty_0\|_{H^1(D)},
\end{equation}
where we recall that $E=y^\perp\cdot \nabla_y$.

Next, we introduce the three-component vector function 
\begin{equation} \label{w0sigma}
\mathbf{w}_0^{\sigma} = \mathbf{w}^\infty_0 + (\mathbf{v}_0^{\sigma},0) \in H^1_0(D),
\end{equation}
which by construction satisfies \eqref{TwistedDivFreeCond}, since
$w^{\sigma,3} = w^{\infty,3}$. 
%Although this function is $\sigma$-independent, we write $\mathbf{v}^\sigma_0$ and $\mw^\sigma_0$ to emphasize that they will be used to obtain a $\sigma$-dependent solution of \eqref{3DNStwistedcart}.
We will take $\mw^\sigma_0$ so constructed as  initial data  for \eqref{3DNStwistedcart}. 
We are now ready to state our first theorem.

\begin{theorem} \label{MainThm1}
Fix $\sigma \geq 1$.
Let $\mw^\infty_0\in H^1_0(D)$ satisfy \eqref{straightdivfree}. Let
$\mw^\infty$ be the unique  strong solution of \eqref{2D3compNS} with initial data  
$\mw^\infty_0$. Let $\mbfu^\infty$ be the unique $2\frac{1}{2}$D
solution of the 
Navier-Stokes equations \eqref{3DNSper} associated to $\mw^\infty$ via \eqref{uinftydef}. 
Let $\mw^\sigma_0$ be given by \eqref{w0sigma} for a choice of
$\mathbf{v}^\sigma_0$ solution of \eqref{correctionu}, and denote by
$\mw^\sigma$ the strong  solution of \eqref{3DNStwistedcart} with initial data 
$\mw^\sigma_0$. Let $\mathbf{u}_0^{\sigma}$ be the associated strong helical solution of the 
Navier-Stokes equations \eqref{3DNSper} given by Proposition \ref{helicalYvariables}. Then, 
for any fixed $T>0$, 
\begin{equation} \label{eq.MainThm1.1}
\begin{aligned}
 \|\mathbf{u}^{\sigma}(t,\cdot,x_3=0) - \mathbf{u}^\infty(t,\cdot,x_3=0)\|_{L^2(D)} &\leq \frac{C}{\sqrt{\sigma}}, \qquad \text{ for all } \; 0<t<T, \\
\|\nabla_H\mathbf{u}^{\sigma}\vert_{x_3=0} - \nabla_H\mathbf{u}^\infty\vert_{x_3=0}\|_{L^2(0,T;L^2(D))} &\leq \frac{C}{\sqrt{\sigma}},
\end{aligned}
\end{equation}
where $C$ is independent of $\sigma \in [1,\infty)$.
\end{theorem}

\begin{proof}
Since by hypothesis, both $\mbfu^\sigma$ and $\mbfu^\infty \in
C([0,T);H^1_{0,per}(\Omega^\sigma)$ \\
$ \cap L^2((0,T);H^2_{per}(\Omega^\sigma)\cap
H^1_{0,per}(\Omega^\sigma))$, 
%and they are smooth for $t>0$ and $x\in \Omega^\sigma$, 
the traces \ $\mbfu^\sigma\vert_{x_3=0}(t)$ and $\mbfu^\infty\vert_{x_3=0}(t)$ 
are well defined as elements of $L^2(D)$ for all $0\leq t<T$, while the traces $\nabla \mbfu^\sigma\vert_{x_3=0}$ and  
$\nabla \mbfu^\infty\vert_{x_3=0}$ are well defined as elements of $L^2((0,T);L^2(D))$.

We continue by showing that
\begin{equation} \label{enough1}
\|\mathbf{w}_0^{\sigma}- \mathbf{w}^\infty_0\|_{L^2(D)} \leq \frac{C}{\sqrt{\sigma}} \;\;\;\mbox{ and}
\end{equation}
\begin{equation} \label{enough2}
\|\mathbf{w}_0^{\sigma}\|_{L^2(D)} \leq C,
\end{equation}
with constants $C$ uniform in $\sigma \in [1,\infty)$.
To see that the first statement \eqref{enough1} holds true, we note that
\[\mathbf{w}_0^{\sigma}- \mathbf{w}^\infty_0 = (\mathbf{v}_0^{\sigma},0),\]
where $\mathbf{v}_0^{\sigma}$ is a solution of  \eqref{correctionu} and satisfies \eqref{correctionuestimate}.

Hence,
\[
\|\mathbf{w}_0^{\sigma}- \mathbf{w}^\infty_0\|_{L^2(D)} = \|\mathbf{v}_0^{\sigma}\|_{L^2(D)} \leq
\|\mathbf{v}_0^{\sigma}\|_{H^1(D)} \leq  \frac{C}{\sigma}\|\mathbf{w}^\infty_0\|_{H^1(D)}.
\]
The second statement \eqref{enough2} follows immediately from the first.

Then, Proposition \ref{EnergyEstDiff} gives that 
\begin{equation} \label{equivw} 
\|\mathbf{w}^{\sigma}(t,\cdot) - \mathbf{w}^\infty(t,\cdot)\|_{L^2(D)} \leq \frac{C}{\sqrt{\sigma}},\;\;\;\mbox{ for almost all } 0<t<T,
\end{equation}
and
\begin{equation} \label{equivnablaw}
\|\nabla_H\mathbf{w}^{\sigma} - \nabla_H\mathbf{w}^\infty\|_{L^2(0,T;L^2(D))} \leq \frac{C}{\sqrt{\sigma}},
\end{equation}
again with constants $C$ that do not depend on $\sigma\geq 1$.

Next, the proof of Proposition \ref{helicalYvariables}
shows that the helical solution $\mbfu^\sigma$ of \eqref{3DNSper} with initial condition $\mbfu^\sigma_0$ related to $\mw^\sigma_0$ via \eqref{wdefweak} is given by 
\[
\mathbf{u}^{\sigma}(t,x',x_3) = M^{\sigma}(x_3)\,\mathbf{w}^{\sigma}(t,m^{\sigma}(x_3) x'),
\]
for $t>0$, where $x'=(x_1,x_2)$, so that in particular:
\[
 \mathbf{u}^{\sigma}(t,x_H)=
 \mathbf{w}^{\sigma}(t, x'), \qquad
  \nabla_H \mathbf{u}^{\sigma}(t,x_H)=
 \nabla_{x'}\mathbf{w}^{\sigma}(t, x').
\]
From \eqref{uinftydef}, it also immediately follows that 
\[
  \mbfu^\infty(t,x_H)= \mw^\infty(t,x'), \qquad
\nabla _x\mathbf{u}^{\infty}(t,x_H)=
   \nabla_H  \mathbf{u}^{\infty}(t,x_H)=
 \nabla_{x'}\mathbf{w}^{\sigma}(t, x').
\]
Then, estimate \eqref{eq.MainThm1.1} is a straightforward consequence of \eqref{equivw} and \eqref{equivnablaw}.
\end{proof}

\begin{rem}
It is natural to derive bounds of traces at $x_3=0$ in view of
\eqref{eq.sigmalim}. In fact, recalling that $\mbfu^\sigma$ is smooth
in $x\in \Omega^\sigma$ for $t>0$, a simple argument, using  a Taylor's
expansion for $\mbfu^\sigma$ in $0\leq x^3/\sigma<1$, centered at $0$ 
with $x'\in D$ fixed, shows that for a given fixed $t$,
\[
   |\mbfu^\sigma(t,x)-\mbfu^\infty(t,x)|=|\mw^\sigma(t,x')-\mw^\infty(t,x')| + O\left(\frac{|x_3|}{\sigma}\right),
\]
with bounds that depend on $|\mw^\sigma(x')|$ and $|\nabla_y \mw^\sigma(x')|$.  Therefore, an argument similar to that of the proof of Theorem \ref{MainThm1} above gives:
\begin{equation}
  \begin{aligned}
    &\|\mathbf{u}^{\sigma}(t) - \mathbf{u}^\infty(t)\|_{L^2(U)}\underset{\sigma\to \infty}{\to 0}, \qquad \text{ for all } \; 0<t<T, \\
&\|\nabla \mathbf{u}^{\sigma} - \nabla \mathbf{u}^\infty\|_{L^2(0,T;L^2(U))} \underset{\sigma\to \infty}{\to 0},
   \end{aligned}
\end{equation}
for any cylinder $U\subset \Omega$ of the form
\[
     U = \{x=(x',x_3) \;|\; x'\in D, \; x_3 \in [0,\delta],  \;
     \delta/\sigma \underset{\sigma\to \infty}{\longrightarrow} 0  \}.
\]
On the other hand, $|x_3|/\sigma$ is $O(1)$ in $\Omega^\sigma$. Hence, it seems difficult to obtain any convergence estimate of $\mbfu^\sigma$ to $\mbfu^\infty$ globally in $\Omega^\sigma$ as $\sigma \to \infty$.
\end{rem}

The previous result is not exactly what we aimed at, as it represents a way of approximating a general two-dimensional flow by a well-chosen helical flow. What we want, instead, is to show that helical flows with large $\sigma$ are nearly two-dimensional. This adjustment is expressed in our next result.

\begin{theorem} \label{MainThm2}
Fix $\sigma\geq 1$ and  $T>0$. Let $\mathbf{u}_0^{\sigma} \in H^1_{0,per}(\Omega^{\sigma})$ be a divergence-free, helical vector field.
Let $\mathbf{u}^{\sigma}$ be the unique, strong helical  solution of
\eqref{3DNSper} on $[0,T)$ with initial velocity $\mathbf{u}_0^{\sigma}$. 
There exists a (not necessarily unique) $\widetilde{\mw_0^{\infty}} 
\in H^1(D)$,   such that, if
$\widetilde{\mbfu^\infty}$ is the unique 
$2\frac{1}{2}$D solution of the Navier-Stokes equations \eqref{3DNSper} with
initial data $\widetilde{\mbfu^\infty_0}(\cdot,x_3) = \widetilde{\mw^\infty}_0$, then  for
all $0<t<T$,
\begin{equation} \label{eq.MainThm2}
 \begin{gathered}
 \|\mathbf{u}^{\sigma}(t,\cdot,x_3=0) - \widetilde{\mathbf{u}^{\infty}}(t,\cdot,x_3=0)\|_{L^2(D)}
 + \qquad \qquad \qquad  
\\ \qquad \qquad \qquad \|\nabla_H \mathbf{u}^{\sigma}\vert_{x_3=0} - \nabla_H \widetilde{\mathbf{u}^{\infty}}\vert_{x_3=0}
 \|_{L^2(0,T;L^2(D))}  
 \leq C(T)
\frac{1}{\sqrt{\sigma}},
 \end{gathered}
\end{equation}
where $C$ is independent of $\sigma \in [1,\infty)$.
\end{theorem}

We use the notation $\widetilde{\mbfu^\infty}$ to emphasize that,
while this is a solution of 
the limit problem, it is still dependent on $\sigma$ due to the
correction to the initial 
condition to enforce the divergence-free condition.

\begin{proof}
As in the proof of Theorem \ref{MainThm1}, the traces of $\mbfu^\sigma$ and $\mbfu^\infty$ are well defined at the level of strong solutions.
Furthermore, as in that theorem we will introduce a correction to the
initial data $\mbfu^\sigma_0$ to enforce the divergence-free condition
on the initial data $\widetilde{\mbfu^\infty}_0$ we take for the limit problem. 
Let $\mw^\sigma_0\in H^1_0(D)$ be  associated to
the helical field $\mbfu^\sigma_0\in H^1_{0,per}(\Omega^\sigma)$ by
\eqref{wdefweak}, satisfying \eqref{TwistedDivFreeCond}. Let
$\mw^\sigma$ be the regular solution of \eqref{3DNStwistedcart} with
this initial data.
%, which by Proposition \ref{helicalYvariables} is given by \eqref{uandpinY} for $t>0$.

Next, let $\mathbf{v}_0^{\sigma}=(v_0^{\sigma,1},v_0^{\sigma,2}) \in H^{1}_0(D)$ be a solution, given by Lemma \ref{GaldisLemma},
to the problem
\begin{equation} \label{correctionw}
\mbox{div}_y \mathbf{v}_0^{\sigma} = -\frac{2\pi}{\sigma}Ew_0^{\sigma,3},
\end{equation}
where again $E$ is the differential operator defined in \eqref{E1},
such that
\begin{equation} \label{correctionwestimate}
\|\mathbf{v}_0^{\sigma}\|_{H^1} \leq C\frac{1}{\sigma}\|E \,w_0^{\sigma,3}\|_{L^2(D)}\leq \frac{C}{\sigma}\|\mathbf{w}_0^{\sigma}\|_{H^1(D)}.
\end{equation}
Its existence is justified exactly as before.

We then set
\begin{equation} \label{tildeu0sigma}
\widetilde{\mathbf{u}_0^{\infty}}(x) = \widetilde{\mathbf{w}^\infty}_0(x'):= \mathbf{w}_0^{\sigma} (x')- (\mathbf{v}_0^{\sigma}(x'),0),
\end{equation}
which is divergence free by \eqref{correctionw}.
Let $\widetilde{\mathbf{w}^{\infty}}$ be the solution of
\eqref{2D3compNS} with  initial data $\widetilde{\mw^\infty}_0$. The
$2\frac{1}{2}$D solution of the Navier-Stokes equations is given by 
$\widetilde{\mathbf{u}^{\infty}}(t,x)= 
\widetilde{\mathbf{w}^{\infty}}(t,x')$. In particular, the trace
$\widetilde{\mathbf{u}^{\infty}}(t,\cdot,x_3=0)=\widetilde{w^\infty}(t,\cdot)$.

By Proposition \ref{EnergyEstDiff}, estimate \eqref{eq.MainThm2} now follows from
\begin{eqnarray}
\label{thetatilde} \|\mathbf{w}_0^{\sigma} - \widetilde{\mathbf{w}_0^{\infty}} \|_{L^2(D)} = \|\mathbf{v}_0^{\sigma}\|_{L^2(D)} \leq  \|\mathbf{v}_0^{\sigma}\|_{H^1(D)}
\leq \frac{C}{\sigma}\|\mathbf{w}_0^{\sigma}\|_{H^1(D)}, \\ \nonumber \\
\label{utildest} \|\widetilde{\mathbf{w}_0^{\infty}}\|_{L^2(D)} \leq \|\mathbf{w}_0^{\sigma} \|_{L^2(D)} + \|\mathbf{v}_0^{\sigma}\|_{L^2(D)} \leq \left(1+\frac{C}{\sigma}\right)\|\mathbf{w}_0^{\sigma}\|_{H^1(D)}, \\ \nonumber \\
\label{w0H1est} \|\mathbf{w}_0^{\sigma}\|_{L^2(D)} \leq \|\mathbf{w}_0^{\sigma}\|_{H^1(D)} = \|\mathbf{u}_0^{\sigma}(x_3=0)\|_{H^1(D)},
\end{eqnarray}
with constants uniform in $\sigma\in [1,\infty)$.
\end{proof}

\section{The inviscid case} \label{Euler}

In this section we discuss symmetry reduction and the limit
$\sigma\to\infty$ for the Euler equations under an additional geometric
assumption, considered already in \cite{D,ET}. This assumption
can be viewed as the analog of the no-swirl condition in axisymmetric
flows and for this reason we will call it the {\em no helical swirl} or
{\em no helical stretching} condition.
It can be shown that the flow induced by solutions of the Euler
equations preserves this condition at least when the solution is
regular enough. Furthermore, vorticity  has an especially
simple form, being determined by its vertical component, which is
advected by the flow. This observation allows to prove global
existence and uniqueness of weak, helical solutions in much the same spirit as for
solutions to the two-dimensional Euler equations, provided the initial
velocity is bounded (cf. \cite{Yudovich63}.)

We now briefly review these results, referring the reader to
\cite{D,ET} for more details. We will then discuss the limit problem
as $\sigma\to \infty$ and converge of solutions. On one hand
the limit problem is simpler, being given by the 2D Euler
equations. In fact, under the no-stretching constraint the
symmetry-reduced helical Euler equations becomes a two-dimensional
systems for two components of the velocity, which admits a
vorticity-stream function formulation (see
e.g. \cite{MajdaBertozzi2002}.)  This system is the analog of the
symmetry-reduced equations \eqref{3DNStwistedcart} for the
Navier-Stokes.
On the other hand, to circumvent the lack of smoothing in the equations for positive
time we will use compactness arguments to pass to the limit
in $\sigma$, which do not provide a rate of convergence.

For ease of notation, we temporarily suppress the explicit
$\sigma$-dependence of solutions and write $\mbfu$ for $\mbfu^\sigma$ for
example. 
 We  assume for now that $\mbfu$ and $p$ are smooth, so
that all the manipulations to follow are justified.

Given that smooth, helical vector fields and functions
are $\sigma$ periodic by Proposition \ref{HelicalCartCorr}, we state  the
initial-boundary-value problem for the Euler equations in the
fundamental domain $\Omega^\sigma$:
\begin{equation} \label{3DEEper}
\left\{
\begin{array}{ll}
\partial_t \mathbf{u} + (\mathbf{u}\cdot\nabla)\mathbf{u} = -\nabla p, & \mbox{ in }(0,+\infty) \times \Omega^\sigma; \\
\dive \,  \mathbf{u} = 0, & \mbox{ in } [0,+\infty) \times
\Omega^\sigma;\\
\mathbf{u}(t,x',x_3) \cdot x'= 0, & \mbox{ for  } t \in [0,+\infty), 
\quad |x'|=1, \; 0\leq x_3\leq \sigma;\\
\mathbf{u}(t,x',0) =\mbfu(t,x',\sigma) & \mbox{ for  } t \in [0,+\infty), 
\quad x'\in D;\\
p(t,x',0) =p(t,x',\sigma) & \mbox{ for  } t \in [0,+\infty), 
\quad x'\in D;\\
\mathbf{u}(0,x) = \mathbf{u}_0, & \, x\in\Omega^\sigma,
\end{array}
\right.
\end{equation}
where again $x=(x',x_3)$ and $x'=(x_1,x_2)$.

Let 
\begin{equation} \label{xivectordef}
   \mxi : = \left( x_2, -x_1, \frac{\sigma}{2\pi} \right)  =
   - \mathbf{x}_H^\perp +  \frac{\sigma}{2\pi} \mathbf{e}_3.
\end{equation}
We will consider flows satisfying the following no-helical-swirl or
stretching condition:
\begin{equation} \label{nostretch}
           \mbfu \cdot \mxi =0.
\end{equation}
This condition is preserved by smooth flows under the time evolution
governed by the Euler equations. 

There are several  consequences of this condition.  Firstly, the vertical
component $u_3$ of the velocity field $\mbfu$ is computed from the
other two components, i.e., the dynamics is planar. Secondly,
 the vorticity $\bomega= \curl_x \mbfu$ is given by
\begin{equation} \label{Eulervortdef}
    \bomega(t, x) = \frac{2\pi}{\sigma}  \omega(t,x) \, \mxi,\qquad \omega :=\omega^3,
\end{equation}
where $\omega^3$ is the component of the vorticity along the axis of
the cylinder $\Omega$. Furthermore, $\omega$ is advected by the
flow $\mbfu$:
\begin{equation} \label{Eulervorteq}
   \pa_t \omega + \mbfu\cdot \nabla \omega = 0.
\end{equation}

To derive the symmetry-reduced equations, we recall that $\mw(t,y) =
\mbfu(t,y,0)$ from Proposition \ref{HelicalCartCorr}, given that the matrices
$M$ and $m$ becomes the identity matrix for $x_3=0$. Consequently,
\begin{equation} \label{omegaident}
   \varpi(t,y) := \omega(t,y,0) = -\nabla_y^{\perp} \mw_H(t,y) = \curl_y \mw_H(t,y).
\end{equation}
Above, to avoid introducing further notation, we have abused notation slightly and
identified $(w^1,w^2)$ with $\mw_H= (w^1,w^2,0)$, where $w^1$, $w^2$
are the horizontal  components of $\mw$ with
respect to the standard Cartesian frame in $\RR^3$.

While $\mw_H$ is not divergence free in view of
\eqref{3DNStwistedcart2}, one observes that a divergence-free 2D flow
can be constructed from $\mw$ under the no-helical-swirl condition,
which 
therefore admits a stream
function $\psi$ on $D$. This stream function satisfies:
\begin{equation} \label{psidef}
  \begin{cases}
         \pa_{y_1} \psi  = \frac{4\pi^2}{\sigma^2} \left [ -y_1 y_2 \,
           w^1
         + \left ( \frac{\sigma^2}{4\pi^2}  + y_1^2 \right) 
       w^2  \right], \\
       \pa_{y_2} \psi  = - \frac{4\pi^2}{\sigma^2} \left [ \left (
           \frac{\sigma^2}{4\pi^2}  + y_2^2 \right)  w^1 -y_1 y_2
         \, w^2
          \right ].
   \end{cases}
\end{equation}
We define the following matrix:
\begin{equation} \label{Kinversedef}
  H(y) :=     \frac{4\pi^2}{\sigma^2}   \left [ \begin{matrix}    
        \left ( \frac{\sigma^2}{4\pi^2}  + y_2^2 \right)   &  -y_1 y_2
        \\
  -y_1 y_2   &   \left ( \frac{\sigma^2}{4\pi^2}  + y_1^2
          \right)  \\
       \end{matrix} \right ],
\end{equation}
and  rewrite \label{psidef}  as  
\[
           \nabla_y^\perp \psi = H(y)\, \mw_H.
\]
A direct calculation, as in \cite{ET}, then shows that 
\begin{equation} \label{Kdef}
   \begin{aligned}
     \curl\, \mw_H  &= \dive K(y)\, \nabla_y \psi,   \quad \text{ with}\\
  K(y) &:=     \frac{1}{\frac{\sigma^2}{4\pi^2} + |y|^2}   \left [ \begin{matrix}    
         \left ( \frac{\sigma^2}{4\pi^2}  + y_1^2
          \right)   &   y_1 y_2       \\
            y_1 y_2   & \left ( \frac{\sigma^2}{4\pi^2}  + y_2^2 \right) 
       \end{matrix} \right ].
\end{aligned}
\end{equation}
From \eqref{omegaident} and \eqref{Kdef}, it follows that 
\[
           \varpi = \cL_H\, \psi, 
\]
where the operator $\cL_H$ is defined by:
\begin{equation}
      \cL_H := \dive_y( K(y)\, \nabla_y).
\end{equation}
It is not difficult to show that $\cL_H$ is a second-order, scalar,
strongly elliptic operator. Consequently, $\nabla^2 \cL_H$ is a
singular integral.
 
Next, calculus inequalities  show that the transport equation
\eqref{Eulervorteq} for $\omega$ reduces by helical symmetry (i.e.,
using the correspondence in Proposition \ref{HelicalCartCorr}) to the
following equation for $\varpi$ on $(0,T)\times D$:
\begin{equation*} %\label{vorteq}
    \pa_t \varpi + \mw_H \cdot \nabla_y \varpi +
    \frac{4\pi^2}{\sigma^2}  (y^\perp\cdot \mw_H) E \varpi = 0,
\end{equation*}
where $E$ is again the operator given in \eqref{E1}. Using
\eqref{Kdef}, we can rewrite this equation as an equation for $\varpi$
and $\psi$ only (cf. \cite[Lemma 2.17]{ET}.)
Furthermore, we can choose Dirichlet boundary conditions for $\psi$
from the no-penetration condition for $\mbfu$ as in Corollary 2.16 of \cite{ET}.
Therefore, under the no-helical swirl condition and for sufficiently
regular solutions, the initial-boundary-value \eqref{3DEEper} for the
Euler equations is equivalent to the following symmetry-reduced system:
\begin{subequations} \label{IBVPstreamfunction}
  \begin{align}
    &\pa_t \varpi + \pa_{y_1} \psi \,\pa_{y_2} \varpi - \pa_{y_2} \psi
    \,\pa_{y_1} \varpi =0,   &y\in D, 0<t<T, \label{streamvort1}\\
    & \varpi = \cL_H \psi,   &y\in D, 0<t<T,  \nonumber\\
    & \psi(0,y) = \psi_0(y),   &y\in D, \\
    & \psi\vert_{\pa D} = 0,  &y \in D.   \label{streamvortbc}
   \end{align}
\end{subequations}
Since \eqref{streamvort1} is a transport equation by the
divergence-free vector field $\nabla^\perp_y \psi$, the $L^\infty$
norm of the reduced vorticity $\varpi$ is preserved under the flow.
By \eqref{Eulervortdef} and Proposition \ref{HelicalCartCorr}, the vorticity $\bomega= \curl_x \mbfu$
is preserved under the flow induced by $\mbfu$.  By the Beale-Kato-Maja criterion ( see
e.g. \cite{MajdaBertozzi2002}) then,  smooth helical solutions
of \eqref{3DEEper} are global in time and agree with weak solutions with the same initial data.

We next discuss weak solution. 
Given $\psi_0\in H^1_0(D)\cap H^2(D)$, we call a function $\psi\in
L^1([0,T);H^1_0(D)\cap H^2(D))$ a weak solution of
the above system on $[0,T)$ with initial data $\psi_0$ if, for all test function $\phi\in
C^\infty_c([0,T)\times D)$, $\psi$ satisfies:
\begin{equation} \label{psieqweak}
  \begin{aligned}
   \int_{D} \cL_H \psi_0 \,\phi(0)\, dy  - \int_0^T \int_D \cL_H \psi
   \,\pa_t \phi \, dy\, dt  + \int_0^T \int_D \pa_{y_2} \psi  \cL_H
   \psi \, \pa_{y_1} \phi \, dy\, dt \\  
    - \int_0^T \int_D \pa_{y_1} \cL_H \psi \, \pa_{y_2} \phi \, dy\,
    dt =   0. 
 \end{aligned}
\end{equation}
Ettinger and Titi  \cite{ET} proved that there exists a unique weak solution on
$[0,T)$, for all $T>0$,
provided in addition $\cL_H \psi_0\in L^\infty(D)$. In this case the solution
satisfies $\cL_H \psi \in L^\infty((0,T)\times D)$.

While there is an existence theory for weak solutions of the Euler
equation in three dimensions \cite{DeLS,Wied11}, we will give here a definition
of weak solution to  \eqref{3DEEper} adapted to the geometry of the
problem and amenable to the analysis of the limit $\sigma\to\infty$
(for further discussion on the uniqueness of helical weak solutions,
we refer the reader to \cite{BLNLT12}.)
Let $\psi$ be the unique weak solution of \eqref{IBVPstreamfunction}
with initial condition $\psi_0\in H^1_0(D)\cap H^2(D)$ such that
$\cL_H \psi_0\in L^\infty(D)$. Let  $\mw  =(\mw_H,w^3)$, where $\mw_H$ is 
given in \eqref{Kdef}  and $w^3$ is obtained from $\mw_H$ via
\eqref{nostretch} as 
\[
     w^3= \frac{2\pi}{\sigma} y^\perp\cdot \mw_H.
\]
Let $\mbfu$ be defined from $\mw$ by
\eqref{wdefweak}. We will call $\mbfu$  a weak, helical
solution of \eqref{3DEEper}. 
This definition is justified in view of  the following proposition.

\begin{proposition} \label{helicalYvariablesEuler}
Let $\{\psi_{0,n}\}$ be a sequence of functions converging to  
$\psi_0\in H^1_0(D)\cap H^2(D)$. Let $\psi_n$ be the smooth solution  
of \eqref{streamvort1} with initial data $\psi_{0,n}$. Then, $\psi_n$
converges uniformly on $[0,T)\times D)$ to $\psi$ the unique
weak solution of \eqref{IBVPstreamfunction}.
\end{proposition}

The proof is contained in \cite{ET}. We recall it briefly.

\begin{proof}
The sequence $\{\psi_n\}$ is uniformly
bounded in $L^1([0,T);H^1_0(D)\cap H^2(D))$ 
and $\cL_H \psi$ is uniformly bounded in
$L^\infty([0,T)\times D)$. 
 Recall  that the equation for $\varpi_n=\cL_H
\psi_n$ is a transport equation by $\nabla^\perp_y \psi_n$, which is
divergence free.
Since $\pa_i\pa_j \cL_H$ is a
Calderon-Zygmund singular integral, $\{\nabla^\perp_y \psi_n\}$ is
bounded in the space LLip of Log-Lipschitz vector fields. Hence, the
family $\{X_n\}$, where $X_n$ is the  flow generated by $\nabla^\perp
\psi_n$ is equicontinuous and hence, upon possibly passing to
subsequences, 
$\varpi_n$ converges strongly in
$L^1((0,T)\times D)$ and $\nabla^\perp_y \psi_n$ converges uniformly
to $\nabla^\perp_y \psi_n$. In particular,  $\psi_n$ converges uniformly to
$\psi$.  These convergence results are enough to pass to the
limit in the weak formulation \eqref{psieqweak}
(cf. \cite[Section 8.2.2]{MajdaBertozzi2002}.) The limit
$\lim_{n\to\infty} \psi_n$ must necessarily agree
 with $\psi$ by uniqueness of the
solution, so the whole sequence converges to $\psi$. 
\end{proof}

This result also implies that, if $\varpi(0)\in L^\infty(D)$, then 
$\varpi(t,x)=\varpi(X^{-1}(t,x)$, where $X$ is the flow generated by 
$\nabla^\perp \psi$, is the (unique) weak solution of
\eqref{IBVPstreamfunction}, hence all its $L^p$ norms are constant in time.

We next discuss the limit $\sigma\to
\infty$. We reinstate the explicit dependence on $\sigma$, and write for
example $\mbfu^\sigma$ for the solution of \eqref{3DEEper},
$\bomega^\sigma$ for $\curl_x  \mbfu^\sigma$ and so on.
We denote the corresponding quantities in the limit by $\mbfu^\infty$,
$\bomega^\infty$ and so on.

Formally taking the limit $\sigma\to \infty$ in
\eqref{nostretch} gives $u^{\infty,3}\equiv 0$ and, hence,
$\mbfu^\infty = \mbfu^\infty_H$. Furthermore, $\mbfu^\infty$ becomes independent of the
$x_3$ variable, so that 
\[
\mbfu^\infty(x) = \mw^\infty(x')=\mw^\infty_H(x')
\]
is divergence-free as a vector field on $D$.
Also, the matrix $K^\sigma$
approaches the identity matrix in the limit, so that $\cL^\infty_H$
is simply 
 the Laplace operator,  $\psi^\infty$ is the stream functions associated to
$\mbfu^\infty_H$, and $\varpi^\infty(x') = \omega^\infty(x) = \curl_{x'} \mbfu^\infty_H(x)$.
In particular,  \eqref{streamvort1}
becomes the vorticity-stream function formulation of the 2D Euler
equations. We conclude that, at least formally, helical solutions to the 3D Euler
equations become planar 2D solutions of the Euler equations as
$\sigma\to\infty$. 

We explicitly state the limit problem:
\begin{subequations} \label{IBVPlimit}
  \begin{align}
    &\pa_t \varpi^\infty + \pa_{y_1} \psi^\infty \,\pa_{y_2} \varpi^\infty - \pa_{y_2} \psi^\infty
    \,\pa_{y_1} \varpi^\infty =0,   &y\in D, 0<t<T, \label{streamvort12D}\\
    & \varpi^\infty = \Delta_y \psi^\infty ,   &y\in D, 0<t<T,  \nonumber\\
    & \psi^\infty(0,y) = \psi^\infty_0(y),   &y\in D, \\
    & \psi^\infty\vert_{\pa D} = 0,  &y \in D.   \label{streamvortbc2D}
   \end{align}
\end{subequations}

Below we will study convergence of the
corresponding stream functions $\psi^\sigma \to  \psi^\infty$ as
$\sigma \to \infty$. 
Since the uniqueness and regularity of weak solutions depends on an
$L^\infty$ control on the vorticity, we will prescribe
the initial vorticity $\varpi^\sigma_0$ independent of $\sigma$, i.e., 
\[
   \varpi^\infty_0 = \varpi^\sigma_0=\varpi_0 \in L^\infty(D).
\]
This choice n can be relaxed by taking a sequence  $\varpi^\sigma_0$
converging to $\varpi_0$ strongly in $L^\infty(D)$.   
We then obtain an initial
condition for the stream function, $\psi_0^\sigma$, that is
$\sigma$-dependent.   
We choose the initial data for the stream function as the unique
solution in $H^1_0(D)$ of the following problems, respectively:
\begin{equation} \label{psi0def}
 \begin{aligned}
       \Delta \psi^\infty_0 =  \varpi_0, \\
        \cL^\sigma_H \psi^\sigma_0 =  \varpi_0.
 \end{aligned}
\end{equation} 
By elliptic regularity, $\psi_0^\infty$, $\psi^\sigma_0 \in W^{2,p}$
for all $1<p<\infty$. 

Next we will derive uniform bounds in $\sigma$ 
on the $W^{2,p}$ norm of  $\psi^\sigma$ and then use compactness
 arguments to pass to
the limit.
It is well known that, under the condition that the initial vorticity
$\varpi_0$ is bounded, solutions to the 2D Euler equations are global
in time and unique \cite{Yudovich63}. 
Therefore, it will be  enough to establish convergence along subsequences.

\begin{lemma} \label{vortSobolevbound}
%Let $1\leq p<\infty$.
Let $1<p<\infty$ be fixed. Then, there exists a constant $C_p>0$ such
that, for all $\sigma>1$ and for all $f \in W^{2,p}(D)$, 
\begin{equation} \label{vortSobolevbound1}
    \|\cL_H f\|_{L^p(D)}\leq C_p \|f\|_{W^{2,p}(D)}.
\end{equation}
Moreover, there exists a $\sigma_0>1$ and  
a constant $C_p>0$,  independent of $\sigma \in [\sigma_0,\infty)$  such
that 
\begin{equation} \label{vortSobolevbound2}
      \|f\|_{W^{2,p}(D)} \leq C_p\, \|\cL_H f\|_{L^p(D)}.
\end{equation}
\end{lemma}

\begin{proof}
We observe that we can write the matrix $K^\sigma = I_2 + F^\sigma$,
where $I_2$ is the $2\times 2$-identity matrix and  
\[
     F^\sigma (y) =   \frac{1}{ 1+ \frac{4\pi^2 \, |y|^2}{\sigma^2}}  
  \left[   \begin{matrix}   \frac{4\pi^2\, y_1^2}{\sigma^2}   &
     \frac{4\pi^2 \, y_1 y_2 }{\sigma^2}  \\
    \frac{4\pi^2 \, y_1 y_2 }{\sigma^2}  &    \frac{4\pi^2 \, y_2^2}{\sigma^2} 
    \end{matrix} 
    \right].
\]
We have:
\begin{equation} \label{Fbounds}
   \|F^\sigma\|_{L^\infty(D)} \leq  C_1\, \frac{1}{\sigma^2}, \qquad
    \|\nabla_y F^\sigma\|_{L^\infty(D)} \leq C_2 \frac{1}{\sigma^2} ,
\end{equation}
for some constants $C_1$, $C_2$ independent of $\sigma$.
The bound \eqref{vortSobolevbound1} then follows immediately.

To establish  \eqref{vortSobolevbound2}, we  write 
\[ 
    \Delta_y f = \cL_H f -  F^\sigma(y):\nabla^2 f - 
(\dive_y F^\sigma(y)) \cdot \nabla_y f,
\]
so that from elliptic regularity for the Poisson's problem 
for $1<p<\infty$, H\"older's inequality and \eqref{Fbounds}:
\[
   \begin{aligned}
          \|f\|_{W^{2,p}} &\leq C_p' \, \|\cL_H f\|_{L^p} + 
        \| F^\sigma :\nabla^2 f \|_{L^p} + 
     \|\dive_y F^\sigma \cdot \nabla_y f\|_{L^p} \\
    &\leq C_p'\, \|\cL_H f\|_{L^p} + C_1\,  \frac{1}{\sigma^2} 
        \|\nabla^2_y f\|_{L^p} + C_2\,  \frac{1}{\sigma^2} \|\nabla_y f\|_{L^p},
  \end{aligned}
\]
or equivalently:
\[ 
  (1-(C_1+C_2)/\sigma^2) \| f\|_{W^{2,p}} \leq C_p' \, \|\cL_H f\|_{L^p}.
\]
So, the result follows provided we choose 
\ $\sigma_0 > 1/\sqrt{(C_1+C_2)}$.
\end{proof}

We now state and prove our convergence result for the Euler equations.
We recall that the only difference between the equations at $\sigma$
finite and in the limit is the  equation expressing the relationship
between the vorticity and the stream function.

\begin{theorem} \label{Eulersigmalimit}
Let $\varpi_0\in L^\infty(D)$. Let $\psi^\sigma_0$  and
$\psi^\infty_0$ be given by \eqref{psi0def}.
  Let $\psi^\sigma$ be the unique
weak solution of \eqref{IBVPstreamfunction} with initial data 
$\psi_0^\sigma$. Let $\psi^\infty$ be the
unique weak solution of \eqref{IBVPlimit} with initial data $\psi^\infty_0$.
Then, $\psi^\sigma$ converges to $\psi$ weakly in $L^p([0,T);W^{1,p}(D)).$
\end{theorem}

\begin{proof}
Since the initial vorticity $\varpi_0\in L^\infty(D)$, $\nabla^\perp
\psi^\sigma \in$ LLip$(D)$ with a bound on the Log-Lipschitz norm 
that is uniform in $\sigma$ for $\sigma \in [1,\infty)$  by
\eqref{Fbounds}. Therefore, we have  a  uniform bound  on
$\varpi^\sigma$ in  $L^\infty([0,T)\times D)$, thanks to the transport
equation \eqref{streamvort1}. 
In turn  by \eqref{vortSobolevbound2}, 
 this bound implies a bound on the family $\{\psi^\sigma\}$
of weak solutions of \eqref{IBVPstreamfunction} 
in  all spaces
$L^\infty([0,T);W^{2,p})$, $1<p<\infty$ that is 
 uniform  in $\sigma\geq \sigma_0$ for $\sigma_0$ large
enough. 

Next, we recall the following a priori bound for weak solutions of
\eqref{IBVPstreamfunction}  (se \cite[Lemma 4.2]{ET}):
\[
     \|\pa_t \psi^\sigma\|_{L^\infty([0,T);W^{1,p}(D))} \leq C_p
     \|\cL_H^\sigma \psi^\sigma\|_{L^\infty((0,T)\times D)}\, \|\psi^\sigma\|_{L^\infty([0,T);W^{1,p})},
\]
where $C_p$ is independent of $\sigma$ for $\sigma$ large enough as in
Lemma \ref{vortSobolevbound}. 
Therefore,  $\{\psi^\sigma\}$ is uniformly bounded in Lip$([0,T);W^{1,p}(D))$.
By the Aubin compactness theorem (see e.g. \cite[Lemma 8.4]{CF}) then,
there is a sequence $\{\psi^{\sigma_n}\}$ that converges strongly in
$L^\infty([0,T);W^{1,p})$ to a function $\psi^\infty$. Upon passing to
a subsequence if necessary, one can assume also that 
$\varpi^{\sigma_n}$ converges weakly-$\ast$  in  $L^\infty([0,T)\times
D)$ to a function $\varpi^\infty$ from the uniform bound obtained
above. 
It remains
to show that $\varpi^\infty=\Delta \psi^\infty$ in $L^2(D)$. This result follows
from the identity $\varpi^\sigma=\cL_H \psi^\sigma$, valid for all $\sigma$, and
\eqref{Fbounds}, by writing again $K^\sigma=I_2 +F^\sigma$.
 
As in the proof of Proposition 5.8 in \cite{ET}, these convergence results are
sufficient to show that $\psi^\infty$ and $\varpi^\infty$ satisfy the
weak formulation of the limit problem \eqref{IBVPlimit}.
But weak solutions of the 2D Euler equations are unique if the
vorticity is bounded, hence any converging sequence of
$\{\psi^\sigma\}$ must converge to $\psi^\infty$.
\end{proof}

\section*{Acknowledgments}  
\mbox{M.~L.~F.} and \mbox{H.~N.~L.} wish to thank the University of California at
Riverside, where part of this work was conducted, for their
hospitality. \mbox{M.~L.~F.} is partially supported by Brazil CNPq grant 
303089/2010-5, and 
CNPq fellowship  200434/2011-0.
H.~N.~L.~is partially supported by Brazil CNPq grant  306331/2010-1,
CAPES fellowship   6649/10-6, and FAPERJ grant  E-26/103.197/2012.
A.~M.  would like to thank the
Institute of Mathematics at the Federal University in Rio de Janeiro
for their hospitality and support. A. M.'s~work  is partially supported by  the
US National Science Foundation grants DMS-1009713 and DMS-1009714.
D.~N.'s work is partially supported by the Chinese National Youth grant 11001184.
The work of E.~S.~T.~was supported in part by   the
Minerva Stiftung/Foundation, and the National Science Foundation grants  DMS-1009950,
DMS-1109640 and DMS-1109645.

\bibliographystyle{amsplain}
\bibliography{helical}

\end{document}